\scrollmode

\documentclass{amsart}
\usepackage{amssymb}
\usepackage[all]{xy}

\theoremstyle{plain}
\newtheorem{prop}{Proposition}
\newtheorem{thm}[prop]{Theorem}
\newtheorem{cor}[prop]{Corollary}
\newtheorem{lem}[prop]{Lemma}
\newtheorem{fact}[prop]{Fact}

\newtheorem*{thmA}{Theorem A}
\newtheorem*{corB}{Corollary B}
\newtheorem*{thmC}{Theorem C}
\newtheorem*{thmD}{Theorem D}

\theoremstyle{definition}

\theoremstyle{remark}

\newtheorem{rem}[prop]{Remark}
\newtheorem{example}[prop]{Example}

\numberwithin{prop}{section} 
\numberwithin{ques}{section} 
\numberwithin{equation}{section}
\numberwithin{step}{prop}

\DeclareMathOperator{\op}{op}
\DeclareMathOperator{\hd}{hd}
\DeclareMathOperator{\rad}{rad}
\DeclareMathOperator{\soc}{soc}
\DeclareMathOperator{\ab}{ab}
\DeclareMathOperator{\el}{el}

\DeclareMathOperator{\cts}{cts}
\DeclareMathOperator{\btor}{\mathbf{tor}}

\DeclareMathOperator{\ccd}{cd}
\DeclareMathOperator{\pdim}{prdim}
\DeclareMathOperator{\gr}{gr}

\DeclareMathOperator{\cl}{cl}

\DeclareMathOperator{\tr}{tr}

\DeclareMathOperator{\Hom}{Hom}

\DeclareMathOperator{\Ext}{Ext}

\DeclareMathOperator{\Tor}{Tor}
\DeclareMathOperator{\image}{im}
\DeclareMathOperator{\kernel}{ker}
\DeclareMathOperator{\coker}{coker}

\DeclareMathOperator{\stab}{stab}
\DeclareMathOperator{\iid}{id}

\DeclareMathOperator{\ob}{ob}
\DeclareMathOperator{\mor}{mor}

\DeclareMathOperator{\ifl}{inf}
\DeclareMathOperator{\rst}{res}
\DeclareMathOperator{\inrst}{ires}
\DeclareMathOperator{\prorst}{tres}
\DeclareMathOperator{\coind}{coind}
\DeclareMathOperator{\idn}{ind}
\DeclareMathOperator{\dfl}{def}

\DeclareMathOperator{\cores}{cor}

\DeclareMathOperator{\limone}{\textstyle{\varprojlim^1}_{\!\!i\geq 0}}
\DeclareMathOperator{\hotimes}{\widehat{\otimes}}
\newcommand{\hH}{\hat{H}}
\newcommand{\ca}[1]{\mathcal{#1}}

\newcommand{\F}{\mathbb{F}}
\newcommand{\N}{\mathbb{N}}
\newcommand{\Z}{\mathbb{Z}}
\newcommand{\Q}{\mathbb{Q}}

\newcommand{\dbl}{[\![}
\newcommand{\dbr}{]\!]}
\newcommand{\FpG}{\F_p\dbl G\dbr}
\newcommand{\ZpG}{\Z_p\dbl G\dbr}

\newcommand{\FpbG}{\F_p\dbl \baG\dbr}

\newcommand{\euX}{\mathfrak{X}}
\newcommand{\euY}{\mathfrak{Y}}

\newcommand{\prf}{\mathbf{prf}}
\newcommand{\tor}{\mathbf{tor}}
\newcommand{\cvec}{{}_{\F_p}\prf}
\newcommand{\prfp}{{}_{\Z_p}\prf}
\newcommand{\dis}{\mathbf{dis}}
\newcommand{\pdis}{{}_{\F_p}\dis}
\newcommand{\zpdis}{{}_{\Z_p}\tor}

\newcommand{\Mod}{\mathbf{mod}}
\newcommand{\fmod}{\mathbf{f.g.}{}_{\Z_p}\Mod}
\newcommand{\Zmod}{{}_{\Z}\Mod}
\newcommand{\FGprf}{{}_{\FpG}\prf}
\newcommand{\ZGprf}{{}_{\ZpG}\prf}
\newcommand{\FGdis}{{}_{\FpG}\dis}
\newcommand{\ZGdis}{{}_{\ZpG}\dis}
\newcommand{\prfFG}{\prf_{\FpG}}
\newcommand{\prfZG}{\prf_{\ZpG}}

\newcommand{\MC}{\mathfrak{cMF}}						
\newcommand{\Adf}{\mathfrak{Add}}
\newcommand{\fgPerM}{\mathbf{f.g.perm}}
\newcommand{\PerM}{\mathbf{perm}}
\newcommand{\der}{\partial}

\newcommand{\eps}{\varepsilon}
\newcommand{\triv}{\{1\}}									

\newcommand{\EE}{\mathbf{E}}							
\newcommand{\facr}[4]{#1\overset{#3}{\to}#2\overset{#4}{\to}#1}	

\newcommand{\baG}{\bar{G}}

\newcommand{\baU}{\bar{U}}
\newcommand{\baV}{\bar{V}}

\newcommand{\caA}{\ca{A}}

\newcommand{\caB}{\ca{B}}

\newcommand{\caR}{\ca{R}}
\newcommand{\caM}{\ca{M}}
\newcommand{\bcM}{\bar{\caM}}
\newcommand{\caL}{\ca{L}}
\newcommand{\caO}{\ca{O}}

\newcommand{\ttau}{\tilde{\tau}}
\newcommand{\tG}{\tilde{G}}
\newcommand{\tpi}{\tilde{\pi}}

\newcommand{\tci}{\tilde{c}}

\newcommand{\tmu}{\tilde{\mu}}

\newcommand{\tphi}{\tilde{\phi}}

\newcommand{\tiota}{\tilde{\iota}}
\newcommand{\teps}{\tilde{\eps}}
\newcommand{\teta}{\tilde{\eta}}

\newcommand{\eui}{\mathfrak{i}}
\newcommand{\eut}{\mathfrak{t}}

\newcommand{\bos}{\mathbf{s}}
\newcommand{\bt}{\mathbf{t}}

\newcommand{\boX}{\mathbf{X}}
\newcommand{\tboX}{\check{\boX}}
\newcommand{\hboX}{\hat{\boX}}
\newcommand{\boY}{\mathbf{Y}}
\newcommand{\boZ}{\mathbf{Z}}

\newcommand{\tboY}{\check{\boY}}
\newcommand{\boA}{\mathbf{A}}

\newcommand{\boB}{\mathbf{B}}

\newcommand{\boS}{\mathbf{S}}
\newcommand{\boT}{\mathbf{T}}
\newcommand{\buT}{\boldsymbol{\Upsilon}}
\newcommand{\bok}{\mathbf{k}}
\newcommand{\boc}{\mathbf{c}}

\newcommand{\El}{\mathbf{El}}

\newcommand{\boh}{\mathbf{h}}

\newcommand{\argu}{\hbox to 7truept{\hrulefill}}

\begin{document}
\title[The projective dimension of profinite modules for pro-p groups]{The projective dimension of profinite modules\\
for pro-p groups}
\author{Th. Weigel}
\date{\today}
\address{Th. Weigel\\
Universit\`a di Milano-Bicocca\\
U5-3067, Via R.Cozzi, 53\\
20125 Milano, Italy}
\email{thomas.weigel@unimib.it}

\begin{abstract}
The homology groups introduced by A.~Brumer can be used
to establish a criterion ensuring that a profinite $\FpG$-module
of a pro-$p$ group $G$ has projective dimension $d<\infty$
(cf. Thm.~A). This criterion yields a new characterization of
free pro-$p$ groups (cf. Cor.~B). Applied to a semi-direct factor
$\facr{G}{\Z_p}{\tau}{\sigma}$ isomorphic to $\Z_p$ which defines
a non-trivial end in the sense of A.A.~Korenev one concludes 
that the closure of the normal closure
of the image of $\sigma$ is a free pro-$p$ subgroup (cf. Thm.~C).
From this result we will deduce a structure theorem (cf. Thm.~D)
for finitely generated pro-$p$ groups with infinitely many ends.
\end{abstract}


\maketitle


\section{Introduction}
\label{s:intro}
In an abelian category $\caA$ with enough projective objects the
{\it projective dimension} $\pdim(A)$ of an object $A\in\ob(\caA)$
is the minimal length of a projective resolution of $A$ (cf. \eqref{eq:pdimdef}).
Moreover, it is well known that
\begin{equation}
\label{eq:prdim}
\pdim(A)=\min(\{\,d\in\N_0\mid \Ext^{d+1}_{\caA}(A,\argu)=0\,\}\cup\{\infty\})
\end{equation}
(cf. \cite[Lemma~4.1.6]{weib:hom}).
The main purpose of this paper is to show that in
the abelian category of profinte left $\FpG$-modules for a pro-$p$ group $G$
less information is necessary in order to detect the projective dimension of an object,
and to discuss several consequences of this fact.
The first main result can be stated as follows (cf. Thm.~\ref{thm:dim}).

\begin{thmA}
Let $G$ be a pro-$p$ group, and let $\caB$ be a basis of neighbourhoods of $1\in G$
consisting of open subgroups of $G$. Assume further that $M$ is a profinite left $\FpG$-module
satisfying $H_d(G,M)\not=0$.
Then the following are equivalent.
\begin{itemize}
\item[(i)] $\pdim(M)=d$;
\item[(ii)] $\cores_{G,U}\colon H_d(G,M)\longrightarrow H_d(U,\rst^G_U(M))$ is injective
for all $U\in\caB$.
\end{itemize}
\end{thmA}

Here $H_\bullet(G,\argu)$ denote the {\it homology groups} introduced in \cite{brum:pseudo},
and the mapping $\cores_{G,U}\colon H_d(G,M)\longrightarrow H_d(U,\rst^G_U(M))$ is the {\it corestriction
mapping}. For a pro-$p$ group $G$ we denote by $\Phi(G)=\cl(G^p[G,G])$ 
its {\it Frattini subgroup},
and define $G^{\ab,\el}=G/\Phi(G)$ to be its {\it maximal elementary abelian quotient}.
For $d=1$ and the trivial left $\FpG$-module $\F_p$ one deduces the following
consequence of Theorem~A (cf. Cor.~\ref{cor:thdim}).

\begin{corB}
Let $G$ be a pro-p group, and 
let $\caB$ be a basis of neighbourhoods of $1\in G$
consisting of open subgroups of $G$.
Then the following are equivalent:
\begin{itemize}
\item[(i)] $G$ is a free pro-$p$ group;
\item[(ii)] $\tr_{G,U}\colon G^{\ab,\el}\to U^{\ab,\el}$ is 
injective for all $U\in\caB$;
\item[(iii)] the canonical map $j_{\El}\colon G^{\ab,\el}\to D_1(\F_p)$
is injective.
\end{itemize}
\end{corB}

Here $\tr_{G,U}\colon G^{\ab,\el}\to U^{\ab,\el}$ denotes the {\it transfer} from $G$ to $U$, and
\begin{equation}
\label{eq:tateD}
D_1(\F_p)=
\textstyle{\varinjlim_{U\subseteq_\circ G} H^1(U,\F_p)^\vee\simeq 
\varinjlim_{U\subseteq_\circ G} H_1(U,\F_p)\simeq
\varinjlim_{U\subseteq_\circ G} U^{\ab,\el}},
\end{equation}
is the abelian group introduced by J.~Tate in his letter to J-P.~Serre (cf. \cite[\S I, App.~1]{ser:gal}).
Note that the inverse limits in \eqref{eq:tateD} are taken over all open subgroups of $G$
and the maps are given by transfer.
The map $j_{\El}\colon G^{\ab,\el}\to D_1(\F_p)$ is just the canonical map.

Let $G$ be a pro-$p$ group, and let $\tau\colon G\to\Z_p$ be a 
surjective (continuous) homomorphism of pro-p groups.
Since $\Z_p$ is a free pro-$p$ group, there exists a continuous section
$\sigma\colon\Z_p\to G$, i.e., $\tau\circ\sigma=\iid_{\Z_p}$.
For short we call a surjective homomorphism $\tau\colon G\to \Z_p$ with
a section $\sigma\colon\Z_p\to G$ a {\it semi-direct factor} isomorphic to $\Z_p$.
A semi-direct factor $\facr{G}{\Z_p}{\tau}{\sigma}$ isomorphic to $\Z_p$ will be called 
an {\it $\F_p$-direction}, if
$j_{\El}(\sigma(1))\not=0$.
For a pro-$p$ group with an $\F_p$-direction one has the following (cf. Thm.~\ref{thm:Fpdir}).

\begin{thmC}
Let $G$ be a pro-$p$ group, and let $\facr{G}{\Z_p}{\tau}{\sigma}$ be
an $\F_p$-direction. Then $N=\cl(\langle\,{}^g\Sigma\mid g\in G\,\rangle)$,
$\Sigma=\image(\sigma)$, is a free normal subgroup, and one has an
isomorphism
$N^{\ab,\el}\simeq\F_p\dbl G/N\dbr$ of profinite left $\FpG$-modules.
Moreover, if $G$ is countably
based, then the induced extension
\begin{equation}
\label{eq:specext}
\xymatrix{
\triv\ar[r]&N^{\ab,\el}\ar[r]&G/\Phi(N)\ar[r]&G/N\ar[r]&\triv
}
\end{equation}
of pro-$p$ groups splits.
\end{thmC}

An extension of pro-$p$ groups
\begin{equation}
\label{eq:specext2}
\xymatrix{
\triv\ar[r]&N\ar[r]^{\iota}&G\ar[r]^{\pi}&\baG\ar[r]&\triv
}
\end{equation}
$\baG=G/N$, will be called {\it special}, if
\begin{itemize}
\item[(i)] $N$ is a free pro-$p$ group;
\item[(ii)] $N^{\ab,\el}$ is a projective profinite left $\FpbG$-module;
\item[(iii)] the induced extension of pro-$p$ groups
\begin{equation}
\label{eq:specext3}
\xymatrix{
\triv\ar[r]&N^{\ab,\el}\ar[r]^-{\iota_\circ}&G/\iota(\Phi(N))\ar[r]^-{\pi_\circ}&\baG\ar[r]&\triv
}
\end{equation}
splits.
\end{itemize}
If the special extension \eqref{eq:specext2} of pro-$p$ groups splits,
then one has an isomorphism $G\simeq F\coprod\baG$ for some free
pro-$p$ group $F$, where $\coprod$ denotes the free product in the category
of pro-$p$ groups (cf. Thm.~\ref{thm:comfree}).
Although the authors do not know of any
non-split special extension of pro-$p$ groups,
the authors cannot prove their non-existence either.

In \cite{kor:ends}, A.A.~Korenev initiated a theory of ends
for pro-$p$ groups. For a finitely generated pro-$p$ group $G$
the {\it number of ends} is given by
\begin{equation}
\label{eq:endsdef}
\EE(G)=
\begin{cases}
1+\dim(D_1(\F_p))&\qquad\text{for $|G|=\infty$,}\\
\hfil 0\hfil&\qquad \text{for $|G|<\infty$.}
\end{cases}
\end{equation}
In \cite{wz:stall}, we will modify this definition slightly.
In order to distinguish between the number of ends defined in \cite{wz:stall}
and the number $\EE(G)$, we will call $\EE(G)$ 
from now on the {\it number of $\F_p$-ends} of $G$.
A.A.~Korenev showed that the only possible values of $\EE(G)$ are $0$, $1$, $2$ or $\infty$,
Moreover, $\EE(G)=0$ holds if, and only if, $G$ is finite, and $\EE(G)=2$ holds if, and only if,
$G$ is infinite and virtually cyclic. However, as he mentioned, an analogue of 
Stallings' decomposition theorem (cf. \cite{stall:ends}) is missing in this context.
The existence of non-split special extensions of pro-$p$ groups
might be an obstacle for proving such a theorem in this context.
Nevertheless, from Theorem~C one conludes the following result (cf. Thm.~\ref{thm:infends}).

\begin{thmD}
Let $G$ be a finitely generated pro-$p$ group with $\EE(G)=\infty$.
Then there exists an open normal subgroup $U$ which is a special extension
of rank $1$, i.e., $U$ has a closed normal subgroup $N$, $N\not=U$, such that
\begin{equation}
\label{eq:specext4}
\xymatrix{
\triv\ar[r]&N\ar[r]^{\iota}&U\ar[r]^{\pi}&U/N\ar[r]&\triv
}
\end{equation}
is a special extension, and $N^{\ab,\el}$ is isomorphic to $\F_p\dbl U/N\dbr$ as
profinite left $F_p\dbl U/N\dbr$-module.
\end{thmD}

In \cite{wz:stall} we will introduce the number of {\it $\Z_p$-ends}.
It will be shown that in this context there holds an analogue of Stalling's decomposition theorem
(cf. \cite[Thm.~B]{wz:stall}).
The proof of this decomposition
theorem will make use of Theorem C in an essential way.

The proof of Theorem~A and Theorem~C is based on several facts 
concerning {\it cohomological Mackey functors} for profinite groups
and, in particular, pro-$p$ groups. Since these techniques have never been
applied before in this context, a considerable part of this paper is dealing
with the machinery of cohomological Mackey functors.
The section cohomology groups introduced in section~\ref{s:seco}
will allow us to establish certain injectivity criteria for natural transformations
of cohomological Mackey functors for pro-$p$ groups (cf. \S\ref{ss:injcrit}).
These injectivity criteria will be an essential tool in the proof of Theorem C.


\section{Profinite modules}
\label{s:profmod}
For a prime number $p$ we denote by $\F_p$ the finite field
with $p$ elements, by $\Z_p$ the ring of {\it $p$-adic integers},
and by $\Q_p$ the field of {\it $p$-adic numbers}.

\subsection{Abelian pro-$p$ groups}
\label{ss:abpro}
By $\cvec$ and $\prfp$ we denote the abelian categories of profinite $\F_p$-vector spaces
and abelian pro-$p$ groups, respectively. These categories are full subcategories of the category of 
topological $\Z_p$-modues.
The Pontryagin duality functor $\argu^\vee=\Hom_{\Z_p}^{\cts}(\argu,\Q_p/\Z_p)$
is exact and
induces natural equivalences
\begin{equation}
\label{eq:pontryagin}
\begin{aligned}
\argu^\vee\colon&\cvec^{\op}\longrightarrow\pdis,\qquad&\argu^\vee\colon&\pdis^{\op}\longrightarrow\cvec,\\
\argu^\vee\colon&\prfp^{\op}\longrightarrow \zpdis,&\argu^\vee\colon&\zpdis^{\op}\longrightarrow\prfp,
\end{aligned}
\end{equation}
where $\pdis$ is the abelian category of $\F_p$-vector spaces, and $\zpdis$
denotes the abelian category of discrete $\Z_p$-torsion modules.
By $\fmod$ we denote the full subcategory of $\prfp$ the objects of which
are finitely generated $\Z_p$-modules.
From Pontryagin duality one concludes easily that $M\in \ob(\prfp)$ is projective if, and only if,
$M$ is torsion free. By $\btor(M)\subseteq M$ we denote the closure of the $\Z_p$-torsion elements in $M$,
i.e., $M$ is torsion free if, and only if, $\btor(M)=0$.

For $M\in \ob(\prfp)$ we put $M_{/p}=M/pM$. 
Then $\argu_{/p}\colon\prfp\to\cvec$
is an additive right exact functor, and
one has the following straightforward fact.

\begin{fact}
\label{fact:elemab}
Let $\phi\colon M\to Q$ be a continuous homomorphism of
abelian pro-$p$ groups. 
Then $\phi$ is surjective if, and only if, $\phi_{/p}\colon M_{/p}\to Q_{/p}$ is surjective.
\end{fact}

Let $A\in\ob(\prfp)$ be an abelian pro-$p$ group, and let
$\gr_\bullet(A)$ denote the graded $\F_p[t]$-module
associated to the p-adic filtration $(p^k.A)_{k\geq 0}$.
Then every homogeneous component $\gr_k(A)$ is
a compact $\F_p$-vector space, i.e., $\gr_k(A)\in\ob(\cvec)$ for all $k\geq 0$.
Moreover, $A$ is torsion free if, and only if, $\gr_\bullet(A)$
is a free $\F_p[t]$-module.

Let $\phi\colon A\to B$ be a homomorphism of 
abelian pro-$p$ groups.
Then $\phi$ induces a homomorphism of $\F_p[t]$-modules
$\gr_\bullet(\phi)\colon\gr_\bullet(A)\to\gr_\bullet(B)$, and
$\gr_k(\phi)\colon\gr_k(A)\to\gr_k(B)$ is continuous.
One has the following fact.

\begin{fact}
\label{fact:elemab2}
Let $\phi\colon A\to B$ be a homomorphism of finitely generated, torsion free
$\Z_p$-modules.
Then $\phi$ is split-injective if, and only if, $\gr_0(\phi)\colon\gr_0(A)\to
\gr_0(B)$ is injective. 
\end{fact}

\subsection{Profinite modules for profinite groups}
\label{ss:profprof}
Let $G$ be a profinite group. Then
\begin{equation}
\label{eq:compgralg}
\FpG=\textstyle{\varprojlim_{U\triangleleft_\circ G}\F_p[G/U]}\qquad\text{and}\qquad
\ZpG=\textstyle{\varprojlim_{U\triangleleft_\circ G}\Z_p[G/U]},
\end{equation}
where the inverse limits are running over all open normal subgroups of $G$,
will denote the {\it completed $\F_p$-algebra} and {\it completed $\Z_p$-algebra}
of $G$, respectively. 
By $\FGprf$ we denote the abelian category of
profinite\footnote{In this context the word profinite
has to be read as pro(finite and discrete).} left
$\FpG$-modules, and by $\ZGprf$ the abelian category of
profinite left $\ZpG$-modules.
Note that $\FGprf$ is the Pontryagin dual of  
$\FGdis$, the abelian category of discrete left $\FpG$-modules,
and that $\ZGprf$ is the
Pontryagin dual of $\ZGdis$, the abelian category of discrete
left $\ZpG$-modules. The categories $\FGdis$ and $\ZGdis$
are closed with respect to direct limits. Moreover, the proof
given in \cite[\S III.11]{mcl:hom} can be transferred verbatim in order to show
that they admit minimal injective envelopes. Thus by Pontryagin duality one has
the following.

\begin{fact}
\label{fact:minproj}
The abelian categories $\FGprf$ and $\ZGprf$ are closed with respect
to inverse limits and admit minimal projective covers.
\end{fact}

Let $M\in\ob(\FGprf)$ be a profinite left $\FpG$-module, and let $\rad(M)$ denote
the intersection of all maximal closed submodules of $M$. Then $\hd(M)=M/\rad(M)$
- the {\it head of $M$} - is the Pontryagin dual of the socle of the Pontryagin dual of $M$.
Moreover, $\hd\colon \FGprf\to\FGprf$ is an additive right exact functor.
By the usual standard argument (cf. \cite[proof of Cor. 2.5.4]{ben:coh1}) one concludes the following.

\begin{fact}
\label{fact:trivdiff}
Let $M\in\ob(\FGprf)$, and let
$(P_\bullet,\der_\bullet,\eps_M)$, $\eps_M\colon P_0\to M$, be a minimal projective
resolution of $M$ in $\FGprf$. Then $(\hd(P_\bullet),\hd(\der_\bullet))$ is a chain complex
with $0$-differentials, i.e., $\hd(\der_k)=0$ for all $k\in\Z$.
\end{fact}

Let $(P_\bullet,\der_\bullet,\eps_M)$, $\eps_M\colon P_0\to M$,
be a projective resolution of $M$ in $\FGprf$. The minimal number $n\in\N_0\cup\{\infty\}$
such that $P_{n+j}=0$ for all $j\geq 1$ is called the {\it length} of $(P_\bullet,\der_\bullet,\eps_M)$
and will be denoted by $\ell(P_\bullet,\der_\bullet,\eps_M)$.
The {\it projective dimension} of $M$ is defined by
\begin{equation}
\label{eq:pdimdef}
\pdim(M)=\min\{\,\ell(P_\bullet,\der_\bullet,\eps_M)\mid (P_\bullet,\der_\bullet,\eps_M)\ 
\text{projective resolution of $M$}\,\}.
\end{equation}
This number coincides with the length of a minimal projective resolution of $M$.
For further details concerning profinite $\FpG$- and $\ZpG$-modules
the reader may consult \cite{brum:pseudo}, \cite{NSW:cohnum}, \cite{ribzal:prof} and \cite{weisym:hori}.

\subsection{Homology}
\label{ss:homo}
Let $\prfFG$ and $\prfZG$ denote the abelian categories
of right profinite $\FpG$-modules and right profinite $\ZpG$-modules,
respectively. The {\it completed tensor product}
\begin{equation}
\label{compten}
\begin{aligned}
\argu\hotimes_G\argu\colon&\prfFG\times\FGprf\longrightarrow\cvec,\\
\argu\hotimes_G\argu\colon&\prfZG\times\ZGprf\longrightarrow\prfp,
\end{aligned}
\end{equation}
introduced by A.~Brumer in \cite{brum:pseudo} is a
right/right-exact additive bifunctor. Their left derived functors
will be denoted by $\Tor^G_k(\argu,\argu)$, $k\geq 0$. 
{\it Homology} with coefficients in the profinite left $\ZpG$-module $M$ is defined
by
\begin{equation}
\label{eq:homodef}
H_\bullet(G,M)=\Tor_{\bullet}^G(\Z_p,M).
\end{equation}
Let $G$ be a pro-$p$ group. Then $\F_p\in\ob(\FGprf)$ - the trivial left $\FpG$-module - 
is the only simple profinite left $\FpG$-module up to isomorphism.
In particular, for any profinite left $\FpG$-module $M$ one has a canonical isomorphism
$\hd(M)\simeq \F_p\hotimes_G M$ of profinite left $\FpG$-modules.
Thus from Fact~\ref{fact:trivdiff} one concludes the following.

\begin{fact}
\label{fact:trivdiff2}
Let $G$ be a pro-$p$ group, let $M\in\ob(\FGprf)$, and let
$(P_\bullet,\der_\bullet,\eps_M)$, $\eps_M\colon P_0\to M$, be a minimal projective
resolution of $M$ in $\FGprf$. Then the chain complex 
$(\F_p\hotimes_G P_\bullet,\F_p\hotimes_G\der_\bullet))$
has $0$-differentials, i.e., $\F_p\hotimes_G\der_k=0$ for all $k\in\Z$.
\end{fact}

\subsection{Continuous cochain cohomology groups}
\label{ss:cts}
Let $G$ be a profinite group. For $M\in\ob(\ZGprf)$ we denote by
$\Ext^k_G(\argu,M)$ the left derived functors of $\Hom_G(\argu,M)$.
Then
\begin{equation}
\label{eq:cts0}
H^k_{\cts}(G,M)=\Ext^k_G(\Z_p,M)
\end{equation}
coincides with the {\it $k^{th}$-continuous cochain cohomology group}
with coefficients in the topological $\ZpG$-module $M$.
These cohomology groups were introduced by J.~Tate in \cite{tate:rel}.
In general, the functors $H^k_{\cts}(G,\argu)$ will not commute with
inverse limits. Nevertheless, one has the following property
for countable inverse limits.

\begin{prop}
\label{prop:cts}
Let $G$ be a profinite group, and let
$M=\varprojlim_{i\geq 0} M_i$ with $M$ and $M_i$, $i\geq 0$,
profinite left $\ZpG$-modules.
Then one has a short exact sequence
\begin{equation}
\label{eq:cts1}
\xymatrix{
0\ar[r]&\limone H^{k-1}_{\cts}(G,M_i)\ar[r] 
&H^k_{\cts}(G,M)\ar[r]&\textstyle{\varprojlim_{i\geq 0} H^k_{\cts}(G,M_i)}\ar[r]&0.
}
\end{equation}
for all $k\geq 1$.
\end{prop}

\begin{proof} Let $\alpha_i\colon M_i\to M_{i-1}$ denote the maps in the inverse system.
As $\ZGprf$ has exact inverse limits, one has a short exact sequence
\begin{equation}
\label{eq:cts2}
\xymatrix{
0\ar[r]&M\ar[r]&\textstyle{\prod_{i\geq 0} M_i}\ar[r]^\delta&
\textstyle{\prod_{i\geq 0} M_i}\ar[r]&0,}
\end{equation}
where $\delta((m_i))_j=m_{j}-\alpha_{j+1}(m_{j+1})$.
Since $\prod$ commutes with (co-)homology, and as
one has a natural isomorphism
\begin{equation}
\label{eq:cts3}
\textstyle{\Hom_G(Q,\prod_{i\geq 0} M_i)\simeq\prod_{i\geq 0}\Hom_G(Q,M_i),}
\end{equation}
the long exact sequence associated to \eqref{eq:cts2}
can be written as
\begin{equation}
\label{eq:cts4}
\xymatrix{
\prod_{i\geq 0}H^{k-1}(G,M_i)\ar[r]&
H^k(G,M)\ar[r]&
\prod_{i\geq 0}H^{k}(G,M_i)\ar[r]^{H^k(\delta)}&
\prod_{i\geq 0}H^{k}(G,M_i)
}
\end{equation}
Since $\kernel(H^k(\delta))=\varprojlim_{i\geq 0} H^k(G,M_i)$ and
as $\coker(H^k(\delta))=\limone H^k(G,M_i)$, this yields the claim.
\end{proof}

\begin{rem}
\label{rem:cts}
(a) Since $M_i$ are abelian pro-$p$ groups,
$H^0_{\cts}(G,M_i)=M_i^G$ are also abelian pro-$p$ groups.
Hence $\limone H^0(G,M_i)=0$, and
\begin{equation}
\label{eq:cts5}
\xymatrix{
H^1_{\cts}(G,M)\ar[r]&\textstyle{\varprojlim_{i\geq 0} H^1_{\cts}(G,M_i)}
}
\end{equation}
is an isomorphism.

\noindent
(b) Suppose that $G$ is a finitely generated profinite group.
Then $|H^1_{\cts}(G,A)|<\infty$ for every finite, discrete left
$\ZGprf$-module $A$. By (a), for every
countably based profinite left $\ZGprf$-module $M$, 
the $\Z_p$-module $H^1_{\cts}(G,M)$ carries 
naturally the structure of an abelian pro-$p$ group.
Moreover, if $\phi\colon M\to Q$ is a continuous homomorphism
of countably based profinite left $\ZGprf$-modules, then
$H^1_{\cts}(\phi)\colon H^1_{\cts}(G,M)\to H^1_{\cts}(G,Q)$ is continuous.
Thus, if $M=\varprojlim_{i\geq 0} M_i$, where $M_i$ are countably based
profinite left $\ZGprf$-modules, then $\limone H^1(G,M_i)=0$.
Hence for such a $\ZGprf$-module $M$ one has an isomorphism
\begin{equation}
\label{eq:cts6}
\xymatrix{
H^2_{\cts}(G,M)\ar[r]&\textstyle{\varprojlim_{i\geq 0} H^2_{\cts}(G,M_i)}
}
\end{equation}
\end{rem}

Let $G$ be a profinite group, and let $\tau\colon \tG\to G$
be an extension of $G$ by some abelian pro-$p$ group $K=\kernel(\tau)$,
i.e., one has a short exact sequence of profinite groups
\begin{equation}
\label{eq:seqh2}
\bos\colon\qquad\xymatrix{
\triv\ar[r]&K\ar[r]^{\iota}&\tG\ar[r]_{\tau}\ar[r]&G\ar[r]\ar@{-->}@/_/[l]_\theta&\triv}
\end{equation}
and $K\in\ob(\ZGprf)$. By \cite[Prop.~I.1.2]{ser:gal}, there exists a section
$\theta\colon G\to\tG$ in the category of profinite sets,
i.e., $\theta$ is a continuous mapping of profinite sets satisfying $\tau\circ\theta=\iid_G$.
Hence using the same procedure as for discrete groups (cf. \cite[\S IV.3]{brown:coh}),
one can show that one has a (canonical) one-to-one correspondence
between the equivalence classes of short exact sequences of profinite groups $[s]$
(cf. \eqref{eq:seqh2}) and elements in the group $H^2_{\cts}(G,K)$.
So, if convenient, we will identity the equivalence class of short exact sequences $[s]$
with its corresponding element in $H^2_{\cts}(G,K)$.
From Remark~\ref{rem:cts}(b) one concludes the following property.

\begin{prop}
\label{prop:splitex}
Let $G$ be a finitely generated profinite group, 
and let $\tau\colon \tG\to G$
be an extension of $G$ by some abelian pro-$p$ group $K=\kernel(\tau)$,
which is a countably based profinite left $\ZGprf$-module.
Let $\{\,K_i\mid i\geq 1\,\}$ be a decreasing countable basis of $K$
consisting of open $\ZGprf$-modules. Suppose that
the induced extensions $\tau_i\colon \tG/K_i\to G$ split for all
$i\geq 1$. Then $\tau\colon\tG\to G$ is a split extension.
\end{prop}

\begin{proof} Let $\rho_i\colon H^2_{\cts}(G,K)\to H^2_{\cts}(G,K/K_i)$
denote the canonical maps.
By Remark~\ref{rem:cts}(b), the map
\begin{equation}
\label{eq:cts7}
\textstyle{\rho=\prod_{i\geq 1}\rho_i\colon H^2_{\cts}(G,K)\longrightarrow
\varprojlim_{i\geq 1} H^2_{\cts}(G,K/K_i)}
\end{equation}
is an isomorphism. Let $\bos$ denote the short exact sequence associated to $\tau$
(cf. \eqref{eq:seqh2}). By hypothesis, $\rho_i([\bos])=0$.
Hence $\rho([s])=0$, and $\tau\colon\tG\to G$ is split surjective.
\end{proof}

\subsection{Profinite modules for finite $p$-groups}
\label{ss:proffinp}
Let $G$ be a finite $p$-group, and let $M\in\ob({}_{\F_p[G]}\prf)$.
Then every element $g\in G$ acts on $M$ by a continuous automorphism. In particular,
$\soc(M)=M^G$ is a closed $\F_p[G]$-submodule of $M$.
Moreover, if $M$ is non-trivial, then $\soc(M)$ is non-trivial. From this property one concludes
the following.

\begin{fact}
\label{fact:soctriv}
Let $G$ be a finite $p$-group, and let $M\in \ob({}_{\F_p[G]}\prf)$ be a non-trivial,
profinite left $\F_p[G]$-module. Let $B\subseteq M$ be a closed $\F_p[G]$-submodule
satisfying $B\cap\soc_G(M)=0$. Then $B=0$.
\end{fact}


\section{Cohomological Mackey functors for profinite groups}
\label{s:mackey}
The notion of {\it Mackey functors} for a finite group was introduced by A.~Dress
in \cite{dress:mac}. The theory of Mackey functors
has been developed and applied to the representation theory 
a finite group by several authors
(cf. \cite{bouc:green}, \cite{jt:mac}, \cite{tw:mac2}, \cite{tw:mac3}, 
\cite{blth:mackey}, \cite{pw:macguide}).
Cohomological Mackey functors satisfy the additional identity (cMF$_7$),
which is responsable for the existence of an intrinsic cohomology theory, the {\it section cohomology}
which will be introduced in section \ref{s:seco}.
The definition of a cohomological Mackey functor for a profinite $G$ is in analogy to
the definition for a finite group. However, new phenomena arise, e.g.,
there are two restriction functors $\prorst$ and $\inrst$, which might be different, and,
even more, the restriction functor $\prorst$ does not have to be exact 
in general (cf. \S\ref{ss:rsts}).
For the convenience of the reader we discuss in this section the principal 
features of cohomological Mackey functors for profinite groups.


\subsection{Mackey systems}
\label{ss:macksys}
Let $G$ be a profinite group. A non-empty set $\ca{M}$ of open subgroups  of $G$
will be called a {\it $G$-Mackey system}, if
\begin{itemize}
\item[(MS$_1$)] for all $g\in G$ and $U\in\ca{M}$ one has ${}^gU=gUg^{-1}\in\ca{M}$, and
\item[(MS$_2$)] for all $U,V\in\ca{M}$ one has $U\cap V\in\ca{M}$.
\end{itemize}
For a profinite group $G$ there are two principal $G$-Mackey systems:
The $G$-Mackey system of all open subgroups
\begin{align}
G^{\sharp}&=\{\,U\subseteq G\mid \text{$U$ an open subgroup of $G$}\,\},\label{eq:macksys1}\\
\intertext{and the $G$-Mackey system of all open, normal subgroups}
G^{\natural}&=\{\,U\subseteq G\mid \text{$U$ an open, normal subgroup of $G$}\,\}.\label{eq:macksys2}\\
\intertext{If $G$ is finite, one has also the $G$-Mackey system}
G^\circ&=\{\triv,G\}.\label{eq:macksys3}
\end{align}
By definition, any non-trivial subset of $G^\natural$ is a $G$-Mackey system.
Such a Mackey system will be called {\it normal}.
Every $G$-Mackey system $\caM$ contains a maximal normal $G$-Mackey system
$\caM^\natural=\caM\cap G^\natural$.
A $G$-Mackey system $\ca{M}$ will be called a {\it $G$-Mackey basis},
if $\bigcap_{U\in\ca{M}} U=\triv$. Hence in this case
$\ca{M}$ forms a basis of neighbourhoods of $1\in G$.
If $\caM$ is a $G$-Mackey system and  $U,V\in\caM$ such that $V$ is normal in $U$,
then we call $(U,V)$ a {\it normal section} in $\caM$.


\subsection{Cohomological Mackey functors}
\label{ss:def}
Let $G$ be a profinite group, and let $\ca{M}\subseteq G^\sharp$ be a $G$-Mackey system.
A {\it cohomological $\ca{M}$-Mackey functor $\boX$}
is a family of abelian groups $(\boX_U)_{U\in \ca{M}}$ together with
group homomorphisms
\begin{equation}
\label{eq:mac1}
\begin{aligned}
i_{U,V}^\boX\colon&\boX_U\longrightarrow \boX_V,\\
t_{V,U}^\boX\colon&\boX_V\longrightarrow \boX_U,\\
c_{g,U}^\boX\colon& \boX_U\longrightarrow\boX_{{}^gU},
\end{aligned}
\end{equation}
for all $U,V\in\ca{M}$, $V\subseteq U$, $g\in G$, which satisfy the following identities:
\begin{itemize}
\item[(cMF$_1$)] $i_{U,U}^\boX=t_{U,U}^\boX=c_{u,U}^\boX=\iid_{\boX_U}$ for all $U\in\ca{M}$
and all $u\in U$;
\item[(cMF$_2$)] $i_{V,W}^\boX\circ i_{U,V}^\boX=i_{U,W}^\boX$ and
$t_{V,U}^\boX\circ t_{W,V}^\boX=t_{W,U}^\boX$ for all $U,V,W\in\ca{M}$ and 
$W\subseteq V\subseteq U$;
\item[(cMF$_3$)] $c_{h,{}^gU}^\boX\circ c_{g,U}^\boX=c_{hg,U}^\boX$ for all $U\in \ca{M}$ and 
$g,h\in G$;
\item[(cMF$_4$)] $i_{{}^gU,{}^gV}^\boX\circ c_{g,U}^\boX=c_{g,V}^\boX\circ i_{U,V}^\boX$ 
for all $U,V\in \ca{M}$ and $g\in G$;
\item[(cMF$_5$)] $t_{{}^gV,{}^gU}^\boX\circ c_{g,V}^\boX=c_{g,U}^\boX\circ t_{V,U}^\boX$ 
for all $U,V\in \ca{M}$ and $g\in G$;
\item[(cMF$_6$)] $i_{U,W}^\boX\circ t_{V,U}^\boX=\sum_{g\in W\setminus U/V}
t_{ {}^gV\cap W,W}^\boX\circ c_{g,V\cap W^g}^\boX\circ i_{V,V\cap W^g}^\boX$, where $W^g=g^{-1}Wg$
for all subgroups $U, V, W\in\ca{M}$ and 
$V,W\subseteq U$;
\item[(cMF$_7$)] $t_{V,U}^\boX\circ i_{U,V}^\boX=|U:V|.\iid_{\boX_U}$ 
for all subgroups $U,V\in\ca{M}$, $V\subseteq U$.
\end{itemize}
From now on we will follow the convention used in \cite{pw:macguide}
and omit the domain in the notation of the morphisms $c_{g,U}$, $g\in G$.
Note that (cMF$_1$) and (cMF$_3$) imply that for $U,V\in\ca{M}$, $V$ normal in $U$,
$\boX_V$ carries naturally the structure of a left $\Z[U/V]$-module. 
In this case (cMF$_6$) reduces to the simple identity
\begin{equation}
\label{eq:mac2}
i_{U,V}^\boX\circ t_{V,U}^\boX= N_{U/V},
\end{equation}
where $N_{U/V}=\sum_{x\in U/V} x$.

A homomorphisms of cohomological $\ca{M}$-Mackey functors $\phi\colon\boX\to\boY$ is a family
of group homomorphisms $\phi_U\colon\boX_U\to\boY_U$, $U\in\ca{M}$, which commute with all the mappings
$i_{.,.}$, $t_{.,.}$ and $c_{g,.}$, $g\in G$.
For any abelian subcategory $\ca{C}\subseteq\Zmod$ we denote by 
$\MC_G(\ca{M},\ca{C})$ the abelian category of all 
cohomological $\ca{M}$-Mackey functors with coefficients in the category $\ca{C}$.

\begin{example}
\label{ex:1}
Let $G$ be a profinite group, let $\ca{M}$ be a $G$-Mackey system, and
let $A$ be an abelian group. 

\noindent
(a) By $\boT=\boT(G,A)$ we denote the cohomological $\ca{M}$-Mackey functor
satisfying $\boT_U=A$, for $U\in\ca{M}$,  $i_{U,V}^\boT=\iid_A$ for $V\in\ca{M}$, 
$V\subseteq U$, 
all mappings $c_g$, $g\in G$, are equal to the identity, and $t_{V,U}^\boT=|U:V|.\iid_{A}$.
It is straightforward to verify that $\boT$ satisfies the axioms (cMF$_1$) - (cMF$_7$).

\noindent
(b) Similarly, one defines the cohomological $\ca{M}$-Mackey functor
$\buT=\buT(G,A)$ by $\buT_U=A$, $U\in\ca{M}$, $t_{V,U}^\boS=\iid_A$, $V\in\ca{M}$, $V\subseteq U$,
all mappings $c_g$, $g\in G$, are equal to the identity, and $i_{U,V}^{\buT}=|U:V|.\iid_A$.
\end{example}


\subsection{Discrete permutation modules}
\label{ss:disperm}
There are three alternative descriptions of cohomological Mackey functors
of a profinite group $G$ in terms of additive functors on certain additive categories.

Let $\caM$ be a $G$-Mackey system, and let $\Omega$ be a finite, discrete left $G$-set. 
We call $\Omega$ a left {\it $(G,\caM)$-set},
if $\stab_G(\omega)\in\ca{M}$ for all $\omega\in\Omega$,
and $\Z[\Omega]$ a finitely generated, left {\it $\Z[G,\caM]$-permutation module}.
Let $\fgPerM(G,\ca{M})$ denote the additive category of finitely generated $\Z[G,\caM]$-permutation
modules.
For $U\in\ca{M}$ and $g\in G$ one has an isomorphism
\begin{equation}
\label{eq:isoperm}
\rho_g^U\colon\Z[G/{}^gU]\longrightarrow \Z[G/U],\qquad
\rho_g^U(xgUg^{-1})=xgU,\qquad x\in G.
\end{equation}
Moreover, if $V\in\caM$, $V\subseteq U$, one has homomorphisms of discrete left $\Z[G]$-modules
\begin{align}
\eui_{V,U}\colon &\Z[G/V]\longrightarrow\Z[G/U],\qquad&\eui_{V,U}(xV)&=xU,\label{eq:homperm1}\\
\eut_{U,V}\colon &\Z[G/U]\longrightarrow\Z[G/V],&\eut_{U,V}(xU)&=\textstyle{\sum_{r\in\caR} xrV},
\label{eq:homperm2}
\end{align}
where $x\in G$ and $\caR\subseteq U$ is any system of coset representatives of $U/V$.
One has a natural equivalence
\begin{equation}
\label{eq:perM}
\Adf(\fgPerM(G,\caM)^{\op},\Zmod)\simeq\MC_G(\caM,\Zmod),
\end{equation}
where $\Adf(.,.)$ denotes the category of (covariant) additive functors,
which is achie\-ved by assigning to the additive functor 
$\tboX\in \ob(\Adf(\fgPerM(G,\caM)^{\op},\Zmod))$ the cohomological $\caM$-Mackey functor 
$\boX\in \ob(\MC_G(\caM,\Zmod))$ satisfying
\begin{equation}
\label{eq:ident}
\boX_U=\tboX(\Z[G/U]),\ 
c_{g,U}^\boX=\tboX(\rho_g^U),\
i_{U,V}^\boX=\tboX(\eui_{V,U}),\ 
t_{V,U}^\boX=\tboX(\eut_{U,V}),
\end{equation}
for $U,V\in\caM$, $V\subseteq U$, and $g\in G$.

The additive category $\fgPerM(G,\caM)$ is naturally equivalent to its opposite
category $\fgPerM(G,\caM)^{\op}$, i.e.,
\begin{equation}
\label{eq:permdual}
\argu^\ast=\Hom_{\Z}(\argu,\Z)\colon\fgPerM^{\op}\longrightarrow\fgPerM
\end{equation}
is a natural equivalence. Hence one has also a natural equivalence
\begin{equation}
\label{eq:perM2}
\Adf(\fgPerM(G,\caM),\Zmod)\simeq\MC_G(\caM,\Zmod)
\end{equation}
by assigning the additive functor
$\hboX\in \ob(\Adf(\fgPerM(G,\caM),\Zmod))$
the cohomological $\caM$-Mackey functor $\boX$ satisfying
\begin{equation}
\label{eq:ident2}
\boX_U=\hboX(\Z[G/U]),\ 
c_{g,U}^\boX=\hboX(\rho_{g^{-1}}^{{}^gU}),\
i_{U,V}^\boX=\hboX(\eut_{U,V}),\ 
t_{V,U}^\boX=\hboX(\eui_{V,U}),
\end{equation}
Let $\PerM(G,\caM)$ denote the category of left $\Z[G,\caM]$-permutation modules,
i.e., $M\in\ob(\PerM(G,\caM))$ if $M$ contains a family of finitely generated
$\Z[G,\caM]$-submodules $M_i\in\ob(\fgPerM(G,\caM))$, $i\in I$, such that
$M=\textstyle{\varinjlim_{i\in I} M_i.}$
Then one has a natural equivalence
\begin{equation}
\label{eq:ident2b}
\Adf(\fgPerM(G,\caM)^{\op},\Zmod)\simeq\Adf^-(\PerM(G,\caM)^{\op},\Zmod),
\end{equation}
where $\Adf^-(\PerM(G,\caM)^{\op},\Zmod)$ denotes the category of additive functors commuting
with inverse limits, i.e., for 
$\tboX\in\ob(\Adf^-(\PerM(G,\caM)^{\op},\Zmod))$ and $M$ as above, one has
$\tboX(M)=\textstyle{\varprojlim_{i\in I} \tboX(M_i).}$
In a similar fashion one obtains a natural equivalence
\begin{equation}
\label{eq:ident3}
\Adf(\fgPerM(G,\caM),\Zmod)\simeq\Adf^+(\PerM(G,\caM),\Zmod),
\end{equation}
where $\Adf^+(\PerM(G,\caM),\Zmod)$ denotes the category of additive functors
commuting with direct limits.

\begin{rem}
\label{rem:pontryagin}
Let $G$ be a profinite group, let
$\caM$ be a $G$-Mackey system, and let
$\boX\in\ob(\MC_G(\caM,\prfp))$. Then $\boX^\vee$ given by (cf. \S\ref{ss:abpro})
\begin{equation}
\label{eq:defpontmac}
(\boX^\vee)_U=(\boX_U)^\vee,\
i_{U,V}^{\boX^\vee}=(t_{V,U}^\boX)^\vee,\
t_{V,U}^{\boX^\vee}=(i_{U,V}^\boX)^\vee,\
c_{g,U}^{\boX^\vee}=(c_{g^{-1},{}^gU}^\boX)^\vee,
\end{equation}
for $U,V\in\caM$, $V\subseteq U$, $g\in G$, defines
a cohomological $\caM$-Mackey functor with values
in the abelian category $\zpdis$, i.e.,
$\boX^\vee\in\ob(\MC_G(\caM,\zpdis))$. It is straightforward
to verify that
Pointryagin duality induces natural equivalences
\begin{equation}
\label{eq:pontmaceq}
\begin{aligned}
\argu^\vee\colon&\MC_G(\caM,\prfp)^{\op}\longrightarrow\MC_G(\caM,\zpdis),\\
\argu^\vee\colon&\MC_G(\caM,\zpdis)^{\op}\longrightarrow\MC_G(\caM,\prfp).
\end{aligned}
\end{equation}
and 
\begin{equation}
\label{eq:pontmaceq2}
\begin{aligned}
\argu^\vee\colon&\MC_G(\caM,\cvec)^{\op}\longrightarrow\MC_G(\caM,\pdis),\\
\argu^\vee\colon&\MC_G(\caM,\pdis)^{\op}\longrightarrow\MC_G(\caM,\cvec).
\end{aligned}
\end{equation}
\end{rem}

\subsection{Basic restriction functors}
\label{ss:resmac}
There are several restriction functors for cohomological Mackey functors
for profinite groups. Some of these functors arise simply by restricting the Mackey system;
others are more complicated (cf. \S\ref{ss:rsts}).

Let $G$ be a profinite group, let $\caL$ and $\caM$ be $G$-Mackey systems,
$\caL\subseteq\caM$, and let $\boX\in\ob(\MC_G(\caM,\Zmod))$. 
Restricting the functor $\boX$ to open subgroups which are contained in $\caL$ yields
a restriction functor
\begin{equation}
\label{eq:res0}
\rst^\caM_\caL(\argu)\colon \MC_G(\caM,\Zmod)\longrightarrow \MC_G(\caL,\Zmod).
\end{equation}
Let $\ca{O}\in\caM$.
Then $\caM(\caO)=\{\,U\in\caM\mid U\subseteq\caO\,\}$ is a $\caO$-Mackey system.
Restricting a functor $\boX\in\ob(\MC_G(\caM,\Zmod))$ to open subgroups $U\in\caM$ 
contained in $\caO$ yields a restriction functor
\begin{equation}
\label{eq:res1}
\rst^\caM_{\caM(\caO)}(\argu)\colon \MC_G(\caM,\Zmod)\longrightarrow \MC_\caO(\caM(\caO),\Zmod).
\end{equation}
Let $N$ be a closed, normal subgroup of $G$, and $\caM_N=\{\,U/N\mid U\in\caM,\ N\subseteq U\,\}$.
Then $\caM_N$ is a $G/N$-Mackey system.
Restricting a functor $\boX$ to subgroups $U\in\caM$
containing $N$ yields a restriction functor
\begin{equation}
\label{eq:res2}
\dfl^\caM_{\caM_N}(\argu)=\rst^\caM_{\caM_N}(\argu)\colon \MC_G(\caM,\Zmod)\longrightarrow \MC_{G/N}(\caM_N,\Zmod)
\end{equation}
which can also be considered as a {\it deflation functor}.
All these restriction functors are additive and exact, and we will use the abbreviation $\rst$
for any composition of these functors. In particular, for $U,V\in\caM$,
$V$ normal in $U$, one has a restriction functor
\begin{equation}
\label{eq:res3}
\rst^\caM_{(U/V)^\circ}(\argu)\colon \MC_G(\caM,\Zmod)\longrightarrow \MC_{U/V}((U/V)^\circ,\Zmod)
\end{equation}
(cf. \eqref{eq:macksys3}).


\subsection{The induction functor}
\label{ss:ind}
Let $G$ be a profinite group, and let $H\subseteq G$ be a closed subgroup of $G$.
For every finitely generated, discrete left $\Z[G]$-permutation module $M$,
$\rst^G_H(M)$ is a finitely generated, discrete left $\Z[H]$-permutation module.
Let $\boX\in\ob(\MC_H(H^\sharp,\Zmod))$. Then 
$\boY=\idn_{H^\sharp}^{G^\sharp}(\boX)$ is the cohomological $G^\sharp$-Mackey functor
satisfying
\begin{equation}
\label{eq:ind1}
\boY_U=\tboX(\rst^G_H(\Z[G/U]))=\hboX(\rst^G_H(\Z[G/U]))
\end{equation}
for $U\in G^\sharp$.

\begin{example}
\label{ex:ind}
Let $H$ be a closed subgroup of the profinite group $G$. 

\noindent
{\rm (a)} Let
$\boT=\boT(H,\Z)\in\ob(\MC_H(H^\sharp,\Zmod))$ be the cohomological
$H^\sharp$-Mackey functor described in Example~\ref{ex:1}(a). Let
$\boY=\idn_{H^\sharp}^{G^\sharp}(\boT)$, and let $U\in G^\sharp$. Then
$\boY_U=\Z[H\backslash G/U]$.
For $V\in G^\sharp$, $V\subseteq U$, let $\pi_{V,U}\colon H\backslash G/V\to
H\backslash G/U$ denote the canonical map. Then for $r,s,g\in G$ one has
\begin{align}
c_{g,U}^\boY(HrU)&=Hrg^{-1}{}^gU,\label{eq:indT1}\\
i_{U,V}^\boY(HrU)&=\sum_{HxV\in\pi^{-1}(HrU)} HxV,\label{eq:indT2}\\
t^\boY_{V,U}(HsV)&=|H\cap {}^sU:H\cap {}^sV|\cdot HsU.\label{eq:indT3}
\end{align}
{\rm (b)} Let
$\buT=\buT(H,\Z)\in\ob(\MC_H(H^\sharp,\Zmod))$ be the cohomological
$H^\sharp$-Mackey functor described in Example~\ref{ex:1}(b),
and $\boZ=\idn_{H^\sharp}^{G^\sharp}(\buT)$.
As before one has for $U\in G^\sharp$ that
$\boZ_U=\Z[H\backslash G/U]$.
For $r,s,g\in G$ one has
\begin{align}
c_{g,U}^\boZ(HrU)&=Hrg^{-1}{}^gU,\label{eq:indS1}\\
i_{U,V}^\boZ(HrU)&=\sum_{HxV\in\pi^{-1}(HrU)} 
|H\cap {}^rU:H\cap {}^xV|\cdot HxV,\label{eq:indS2}\\
t^\boZ_{V,U}(HsV)&= HsU.\label{eq:indS3}
\end{align}
\end{example}

\subsection{The restriction functors $\prorst$ and $\inrst$}
\label{ss:rsts}
Let $H$ be a closed subgroup of the profinite group $G$. The
{\it coinduction functor}
\begin{equation}
\label{eq:coind}
\coind_H^G(\argu)\colon {}_H\dis\longrightarrow {}_G\dis,
\end{equation}
$\coind_H^G(M)=C_H(G,M)$ (cf. \cite[\S I.2.5]{ser:gal}), $M\in ob({}_H\dis)$,
is the right adjoint to the restriction functor $\rst^G_H(\argu)\colon{}_G\dis\to{}_H\dis$.
The adjunction is achieved by the counit and unit
\begin{equation}
\label{eq:counit} 
\begin{aligned}
\eps\colon&\rst^G_H(\coind_H^G(\argu))&\longrightarrow&\ \ \iid_{{}_H\dis},\\
\eta\colon&\iid_{{}_G\dis}&\longrightarrow&\ \ \coind_H^G(\rst^G_H(\argu)),
\end{aligned}
\end{equation}
where $\eps_M(f)=f(1)$, $f\in C_H(G,M)$, $M\in\ob({}_H\dis)$, and 
$\eta_Q(q)(g)=g.q$, $q\in Q$, $g\in G$, $Q\in ob({}_G\dis)$, i.e.,
\begin{equation}
\label{eq:counitid}
\begin{aligned}
\iid_{\rst^G_H(Q)}&=\eps_{\rst^G_H(Q)}\circ\rst^G_H(\eta_Q),\\
\iid_{\coind_H^G(M)}&=\coind(\eps_M)\circ\eta_{\coind^G_H(M)},
\end{aligned}
\end{equation}
(cf. \cite[\S IV.1, Thm.~1]{mcl:cat}).
Let $\boX\in\ob(\MC_G(G^\sharp,\Zmod))$ be a cohomological $G^\sharp$-Mackey functor,
let $\tboX\in\ob(\Adf(\fgPerM(G,G^\sharp)^{\op},\Zmod))$ denote the associated
contravariant additive functor, and $\hboX\in\ob(\Adf(\fgPerM(G,G^\sharp),\Zmod))$ the associated
covariant additive functor. One has the restriction functors
\begin{equation}
\label{eq:rsts}
\begin{aligned}
\prorst^{G^\sharp}_{H^\sharp}(\argu)\colon&\MC_G(G^\sharp,\Zmod)\longrightarrow
\MC_H(H^\sharp,\Zmod),\\
\prorst^{G^\sharp}_{H^\sharp}(\boX)_{H_0}&=\tboX(\coind_{H_0}^H(\Z[H/H_0])),\ H_0\in H^\sharp,\\
\inrst^{G^\sharp}_{H^\sharp}(\argu)\colon&\MC_G(G^\sharp,\Zmod)\longrightarrow
\MC_H(H^\sharp,\Zmod),\\
\inrst^{G^\sharp}_{H^\sharp}(\boX)_{H_0}&=\hboX(\coind_{H_0}^H(\Z[H/H_0])),\ H_0\in H^\sharp;
\end{aligned}
\end{equation}
e.g.,
\begin{equation}
\label{eq:rstsex}
\begin{aligned}
\prorst^{G^\sharp}_{H^\sharp}(\boX)_{H_0}&=\varprojlim_{U\supseteq H_0} (\boX_U, t^\boX_{V,U}),\\
\inrst^{G^\sharp}_{H^\sharp}(\boX)_{H_0}&=\varinjlim_{U\supseteq H_0} (\boX_U, i^\boX_{U,V}).
\end{aligned}
\end{equation}
Note that if $H$ is open in $G$ one has natural isomorphisms
\begin{equation}
\label{eq:rstid}
\prorst^{G^\sharp}_{H^\sharp}(\argu)\simeq\inrst^{G^\sharp}_{H^\sharp}(\argu)\simeq
\rst^{G^\sharp}_{H^\sharp}(\argu)
\end{equation}
(cf. \eqref{eq:res1}). This is the reason why in \cite{tw:mac2} there occurs only one restriction functor.
One has the following.

\begin{prop}
\label{prop:adjoint}
Let $G$ be a profinite group, let $H$ be a closed subgroup of $G$, and
let $\boX\in \ob(\MC_H(H^\sharp,\Zmod))$, $\boY\in \ob(\MC_G(G^\sharp,\Zmod))$.
Then one has natural isomorphisms
\begin{align}
\Hom_{G^\sharp}(\idn_{H^\sharp}^{G^\sharp}(\boX),\boY)&
\simeq \Hom_{H^\sharp}(\boX,\prorst_{H^\sharp}^{G^\sharp}(\boY)),
\label{eq:rightad}\\
\Hom_{G^\sharp}(\boY,\idn_{H^\sharp}^{G^\sharp}(\boX))&
\simeq \Hom_{H^\sharp}(\inrst_{H^\sharp}^{G^\sharp}(\boY),\boX),
\label{eq:leftad}
\end{align}
i.e., $\prorst$ is the right-adjoint of $\idn$, and $\inrst$ is the left-adjoint of $\idn$.
\end{prop}

\begin{proof}
Let $H_0\in H^\sharp$, $U\in G^\sharp$, $\boX\in \ob(\MC_H(H^\sharp,\Zmod))$, 
$\boY\in \ob(\MC_G(G^\sharp,\Zmod))$.
The adjunction \eqref{eq:counit} yields homomorphisms of abelian groups
\begin{equation}
\label{eq:adj1}
\begin{aligned}
\eps_\boX(H_0)\colon& \tboX(\Z[H/H_0])\longrightarrow\tboX(\rst^G_H(\coind^G_H(\Z[H/H_0]))),\\
\eta_\boY(U)\colon&\tboY(\coind^G_H(\rst^G_H(\Z[G/U])))\longrightarrow\tboY(\Z[G/U]),
\end{aligned}
\end{equation}
which are natural in $H_0$ and $U$, respectively. Thus they yield morphisms of
cohomological Mackey functors
\begin{equation}
\label{eq:adj2}
\begin{aligned}
\teps_\boX\colon&\boX\longrightarrow\prorst^{G^\sharp}_{H^\sharp}(\idn_{H^\sharp}^{G^\sharp}(\boX)),\\
\teta_\boY\colon&\idn^{G^\sharp}_{H^\sharp}(\prorst^{G^\sharp}_{H^\sharp}(\boY)) \longrightarrow\boY,
\end{aligned}
\end{equation}
which are natural in $\boX$ and $\boY$, respectively, i.e., 
\begin{equation}
\label{eq:adj2b}
\begin{aligned}
\teps\colon&\iid_{\MC_H(H^\sharp,\Zmod)}\longrightarrow \prorst^{G^\sharp}_{H^\sharp}(\idn_{H^\sharp}^{G^\sharp}(\argu))\\
\teta\colon&\idn^{G^\sharp}_{H^\sharp}(\prorst^{G^\sharp}_{H^\sharp}(\argu))\longrightarrow
\iid_{\MC_G(G^\sharp,\Zmod)}
\end{aligned}
\end{equation}
are natural transformations of functors.
By the first line of \eqref{eq:counitid}, the composite of the maps of cohomological $G^\sharp$-Mackey functors
\begin{equation}
\label{eq:adj3}
\xymatrix@C2truecm{
\idn_{H^\sharp}^{G^\sharp}(\boX)
\ar[r]^-{\idn(\teps_\boX)}&
\idn_{H^\sharp}^{G^\sharp}(\prorst^{G^\sharp}_{H^\sharp}(\idn_{H^\sharp}^{G^\sharp}(\boX)))
\ar[r]^-{\teta_{\idn(\boX)}}&
\idn_{H^\sharp}^{G^\sharp}(\boX)
}
\end{equation}
is the identity on $\idn_{H^\sharp}^{G^\sharp}(\boX)$; while the second line 
of \eqref{eq:counitid} yields that the composite of
\begin{equation}
\label{eq:adj4}
\xymatrix@C1.7truecm{
\prorst_{H^\sharp}^{G^\sharp}(\boY)
\ar[r]^-{\teps_{\prorst(\boY)}}&
\prorst_{H^\sharp}^{G^\sharp}(\idn_{H^\sharp}^{G^\sharp}(\prorst_{H^\sharp}^{G^\sharp}(\boY)))
\ar[r]^-{\prorst(\teta_\boY)}&
\prorst_{H^\sharp}^{G^\sharp}(\boY)
}
\end{equation}
is the identity on $\prorst_{H^\sharp}^{G^\sharp}(\boY)$.
Hence $\teps$ is the unit, and $\teta$ is the counit of the adjuction \eqref{eq:rightad}.
The adjunction \eqref{eq:leftad} can be proved by a similar argument.
\end{proof}

\subsection{From modules to cohomological Mackey functors and vice versa}
\label{ss:modmac}
Let $G$ be a profinite group, and let ${}_G\dis$ denote the abelian category of discrete left
$G$-modules (cf. \cite[\S I.2.1]{ser:gal}). For $M\in \ob({}_G\dis)$ let $\boh^0(M)$ denote the cohomological
$G^\sharp$-Mackey functor given by
\begin{equation}
\label{eq:hup0}
\boh^0(M)_U=M^U,
\end{equation}
where for $U,V\in G^\sharp$, $V\subseteq U$, $i^{\boh_0(M)}_{U,V}$ is the canonical
injection. Let $\caR\subseteq U$ be a system of coset representatives of $U/V$. 
Then $t^{\boh_0(M)}_{V,U}=\sum_{r\in R}r\colon M^V\to M^U$ is given by the {\it transfer}, and
for $g\in G$, $c_{g,U}^{\boh_0(M)}\colon M^U\to M^{{}^gU}$ is given by multiplication with $g$.

Let $p$ be a prime number. For $U\in G^\sharp$ let $\omega_U=\kernel(\Z_p\dbl U\dbr\to\Z_p)$
denote the augmentation ideal of the completed $\Z_p$-algebra of $U$.
Let $Q\in\ob(\ZGprf)$. Then $\boh_0(Q)\in\ob(\MC_G(G^\sharp,\prfp))$ will denote the cohomological
$G^\sharp$-Mackey functor given by
\begin{equation}
\label{eq:hdown0}
\boh_0(Q)_U=Q_U=Q/\omega_U.Q\simeq \Z_p[U\backslash G]\hotimes_G Q.
\end{equation}
For $U,V\in G^\sharp$, $V\subseteq U$, the map
$t_{V,U}^{\boh_0(Q)}\colon Q_V\to Q_U$ is just the canonical map, and
$c_{g,U}^{\boh_0(Q)}\colon Q_U\to Q_{{}^gU}$ is the map induced by multiplication with $g\in G$, while
$i_{U,V}^{\boh_0(Q)}\colon Q_U\to Q_V$ is given by 
\begin{equation}
\label{eq:ihdown}
i_{U,V}^{\boh_0(Q)}(q+\omega_U.Q)=\textstyle{(\sum_{r\in\caR} r^{-1}.q)+\omega_V.Q.}
\end{equation}

\begin{example}
\label{ex:hs}
Let $G$ be a profinite group, and let $H$ be a closed subgroup of $G$.

\noindent
{\rm (a)} Let $\Z_p\in\ob({}_G\dis)$ denote the discrete, left $G$-module isomorphic to $\Z_p$.
Then one has a canonical isomorphism $\boh^0(\Z_p)\simeq\boT(G,\Z_p)$
(cf. Ex.~\ref{ex:1}(a)).

\noindent
{\rm (b)} Let $\Z_p\in\ob(\ZGprf)$ denote the profinite, left $\ZpG$-module isomorhic to $\Z_p$.
Then one has a canonical isomorphism $\boh_0(\Z_p)\simeq\buT(G,\Z_p)$
(cf. Ex.~\ref{ex:1}(b)).

\noindent
{\rm (c)} From Example~\ref{ex:ind} one concludes that one has canonical isomorphisms
\begin{equation}
\label{eq:indTS}
\begin{aligned}
\idn_{H^\sharp}^{G^\sharp}(\boT(H,\Z_p))&\simeq\boh^0(\coind_H^G(\Z_p)),\\
\idn_{H^\sharp}^{G^\sharp}(\buT(H,\Z_p))&\simeq\boh_0(\idn_H^G(\Z_p)),
\end{aligned}
\end{equation}
where $\idn_H^G=\ZpG\hotimes_H\argu\colon {}_{\Z_p\dbl H\dbr}\prf \to\ZGprf$ is the 
{\it profinite induction functor} (cf. \S\ref{ss:profprof}).
\end{example}
For $\boX\in\ob(\MC_G(G^\sharp,\Zmod))$, the abelian group
$\inrst(\boX)=\inrst^{G^\sharp}_{\triv^\sharp}(\boX)$ carries naturally the structure of a discrete, left $G$-module.
Moreover, for $M\in\ob({}_G\dis)$ one has a canonical isomorphism
\begin{equation}
\label{eq:hin}
\inrst(\boh^0(M))\simeq M.
\end{equation}
For $\boY\in\ob(\MC_G(G^\sharp,\prfp))$, the abelian pro-p group
$\prorst(\boY)=\prorst^{G^\sharp}_{\triv^\sharp}(\boY)$ carries naturally the structure of a profinite, left $\ZpG$-module.
Furthermore, for $Q\in\ob(\ZGprf)$ one has a canonical isomorphism
\begin{equation}
\label{eq:pro}
\prorst(\boh_0(Q))\simeq Q.
\end{equation}
Note that
$\boh^0\colon{}_G\dis\to\MC_G(G^\sharp,\Zmod)$ is the
right-adjoint of the exact functor
$\inrst\colon\MC_G(G^\sharp,\Zmod)\to{}_G\dis$;
while
$\boh_0\colon\ZGprf\to\MC_G(G^\sharp,\prfp)$ is the left-adjoint
of the exact functor
$\prorst(\argu)\colon\MC_G(G^\sharp,\prfp)\to\ZGprf$.

\subsection{Homology as cohomological Mackey functor}
\label{ss:hommac}
Let $Q\in\ob(\ZGprf)$, and let $(P_\bullet,\der_\bullet,\eps_Q)$ be a
projective resolution of $Q$ in $\ZGprf$. Then $(\boh_0(P_\bullet),\boh_0(\der_\bullet))$
is a chain complex in $\MC_G(G^\sharp,\prfp)$ concentrated in positive degrees.
The cohomological $G^\sharp$-Mackey functor
\begin{equation}
\label{eq:hommac1}
\boh_k(Q)=H_k(\boh_0(P_\bullet),\boh_0(\der_\bullet))
\end{equation}
is independent of the choice of the projective resolution $(P_\bullet,\der_\bullet,\eps_Q)$, and
one has 
\begin{equation}
\label{eq:hommac2}
\boh_k(Q)_U\simeq \Tor^U_k(\Z_p,\rst^G_U(Q))=H_k(U,\rst^G_U(Q)).
\end{equation}
The homology functor commutes with inverse limits.
Thus from \eqref{eq:rstsex} one concludes the following.

\begin{fact}
\label{fact:tres}
Let $G$ be a profinite group, let $H$ be a closed subgroup of $G$, and let
$Q$ be a profinite $\ZpG$-module. Then for all $k\geq 0$ one has
a canonical isomorphism
\begin{equation}
\label{eq:hommac3}
\prorst^{G^\sharp}_{H^\sharp}(\boh_k(Q))\simeq\boh_k(\rst^G_H(Q)).
\end{equation}
\end{fact}

Note that Fact~\ref{fact:tres} implies that 
$\prorst^{G^\sharp}_{\triv^\sharp}(\boh_k(Q))=0$ for all $k\geq 1$,
and for any $Q\in\ob(\ZGprf)$ which is a torsion free abelian pro-$p$ group.
Moreover, Fact~\ref{fact:tres} can be seen as a Pontryagin dual version of  
the fact that for any discrete left $G$-module $M$ and any closed subgroup $H$ of $G$
one has
\begin{equation}
\label{eq:dislim}
H^k(H,\rst^G_H(M))=\textstyle{\varinjlim_{H\subseteq U} H^k(U,\rst^G_U(M))}
\end{equation}
(cf. \cite[\S I.2.1, Prop.~8]{ser:gal}), where the inverse system is running over all
open subgroup of $G$ containing $H$.


\section{Section cohomology of cohomological Mackey functors}
\label{s:seco}
In this section we discuss an intrinsic cohomology theory one can
define for a single cohomological Mackey functor of a profinite group $G$.
Elementary properties of this theory can be used to study the injectivity
of a homomorphism $\phi\colon\boX\to\boY$ of
cohomological Mackey functors (cf. \S \ref{ss:injcrit}),
and they also provide an effective criterion to calculate the
projective dimension of a profinite, left $\FpG$-module $M$
in case that $G$ is a pro-$p$ group (cf. Thm.~\ref{thm:dim}).
Several issues of these cohomology groups were discussed in
\cite{blth:mackey} and \cite{tw:stand}.

\subsection{Cohomological Mackey functors for finite groups}
\label{ss:cohmacfin}
Let $G$ be a finite group, let $G^\circ=\{G,\triv\}$, and
let $\boX\in ob (\MC_G(G^\circ,\Zmod))$.
Define
\begin{equation}
\label{eq:kcs}
\begin{aligned}
\bok^0(G,\boX)&=\kernel(i^\boX_{G,\triv}),&\bok^1(G,\boX)&=\boX_{\triv}^G/\image(i^\boX_{G,\triv}),\\
\boc_0(G,\boX)&=\coker(t_{\triv,G}^\boX),&\boc_1(G,\boX)&=
\kernel(t_{\triv,G}^\boX)/\omega_G\boX_{\triv},
\end{aligned}
\end{equation}
where $\omega_G=\kernel(\Z[G]\to\Z)$ is the augmentation ideal of the $\Z$-group algebra of $G$.
The following properties were established in \cite[\S 2.4]{tw:stand}.

\begin{prop}
\label{prop:secco}
Let $G$ be a finite group, and let $\boX\in \ob(\MC_G(G^\circ,\Zmod))$.
\begin{itemize}
\item[(a)] The canonical maps yield an exact sequence
\begin{equation}
\label{eq:6term}
\xymatrix@R=3pt{
0\ar[r]&\boc_1(G,\boX)\ar[r]&
\hH^{-1}(G,\boX_{\triv})\ar[r]&
\bok^0(G,\boX)\ar[r]&\ldots\\
\ldots\ar[r]&\boc_0(G,\boX)\ar[r]&
\hH^{0}(G,\boX_{\triv})\ar[r]&
\bok^1(G,\boX)\ar[r]&0
}
\end{equation}
where $\hH^\bullet(G,\argu)$ denote the Tate cohomology groups.
\item[(b)] Let $\xymatrix{0\ar[r]&\boX\ar[r]^\phi&\boY\ar[r]^\psi&\boZ\ar[r]&0}$ be a short exact sequence
of cohomological $G^\circ$-Mackey functors. Then one has exact sequences
\begin{equation}
\label{eq:longk}
\xymatrix@R=3pt{
0\ar[r]&\bok^0(G,\boX)\ar[r]^{\bok^0(\phi)}&\bok^0(G,\boY)\ar[r]^{\bok^0(\psi)}&
\bok^0(G,\boZ)\ar[r]&\ldots\\
\ldots\ar[r]&\bok^1(G,\boX)\ar[r]^{\bok^1(\phi)}&\bok^1(G,\boY)\ar[r]^{\bok^1(\psi)}&
\bok^1(G,\boZ)
}
\end{equation}
and
\begin{equation}
\label{eq:longc}
\xymatrix@R=3pt{
&\boc_1(G,\boX)\ar[r]^{\boc_1(\phi)}&\boc_1(G,\boY)\ar[r]^{\boc_1(\psi)}&
\boc_1(G,\boZ)\ar[r]&\ldots\\
\ldots\ar[r]&\boc_0(G,\boX)\ar[r]^{\boc_0(\phi)}&\boc_0(G,\boY)\ar[r]^{\boc_0(\psi)}&
\boc_0(G,\boZ)\ar[r]&0.
}
\end{equation}
\end{itemize}
\end{prop}


\subsection{Cohomological Mackey functors for profinite groups}
\label{ss:CMprof}
Let $G$ be a profinite group, let $\caM$ be a $G$-Mackey system,
and let $U,V\in\caM$, $V$ normal in $U$.
The restriction functor (cf. \eqref{eq:res3})
\begin{equation}
\label{eq:rstrep}
\rst^\caM_{(U/V)^\circ}(\argu)\colon\MC_G(\caM,\Zmod)\longrightarrow\MC_{U/V}((U/V)^\circ,\Zmod)
\end{equation}
is exact. For $k=0,1$ we put (cf. \S\ref{ss:macksys})
\begin{equation}
\boc_k(U/V,\boX)=\boc_k(U/V,\rst^\caM_{(U/V)^\circ}(\boX)),\quad
\bok^k(U/V,\boX)=\bok^k(U/V,\rst^\caM_{(U/V)^\circ}(\boX)).
\end{equation}
Recall that
\begin{equation}
\label{eq:tirstlim}
\begin{aligned}
\prorst(\boX)&=\textstyle{\varprojlim_{U\in\caM} (\boX_U,t^\boX_{V,U}),}\\
\inrst(\boX)&=\textstyle{\varinjlim_{U\in\caM}(\boX_U,i^\boX_{U,V}).}
\end{aligned}
\end{equation}
In particular, $\inrst(\boX)$ is a discrete left $G$-module, and
\begin{equation}
\label{eq:iresex}
\inrst(\argu)\colon\MC_G(\caM,\Zmod)\longrightarrow {}_G\dis
\end{equation}
is an exact functor. If $G\in\caM$, we denote by
\begin{equation}
\label{eq:jdef}
j_{\boX}\colon \boX_G\to\inrst(\boX)
\end{equation}
the canonical map.
If $\boX\in\ob(\MC_G(\caM,\prfp))$, then $\prorst(\boX)$ is a
profinite left $\ZpG$-module, and
\begin{equation}
\label{eq:tresex}
\prorst(\argu)\colon\MC_G(\caM,\prfp)\longrightarrow \ZGprf
\end{equation}
is an exact functor. For 
$\boX\in\ob(\MC_G(\caM,\Zmod))$ the higher derived functors
\begin{equation}
\label{eq:highlim}
\caR^k\prorst(\boX)=\textstyle{\varprojlim^k_{U\in\caM}(\boX_U,t^\boX_{V,U})}
\end{equation}
$k\geq 1$, might not be trivial (cf. \cite{jens:invlim}).

\subsection{Cohomological Mackey functors of type $H^0$}
\label{ss:cohmacprof1}
Let $G$ be a profinite group, let $\caM$ be a $G$-Mackey system, and let
$\boX$ be a cohomological $\caM$-Mackey functor. 
If $i^\boX_{U,V}\colon\boX_U\to\boX_V$
is injective for all $U,V\in\caM$, $V\subseteq U$,
we call $\boX$ {\it $i$-injective}. 
If $\boY$ is a subfunctor of an $i$-injective, cohomological $\caM$-Mackey functor $\boX$,
then $\boY$ is also $i$-injective.
The following fact is straightforward.

\begin{fact}
\label{fact:vank0}
Let $G$ be a profinite group, and let $\caM$ be a $G$-Mackey system.
The cohomological $\caM$-Mackey functor $\boX$ is $i$-injective if, and only if,
$\bok^0(U/V,\boX)=0$ for every normal section $(U,V)$ in $\caM$.
\end{fact}

For profinite $\FpG$-modules (resp. $\ZpG$-modules) of finite
projective dimension one has the following.

\begin{fact}
\label{fact:pdim}
Let $G$ be a profinite group, and let $M$ be
a profinite, left $\FpG$-module (resp. $\ZpG$-module)
of projective dimension $d<\infty$.
Then $\boh_d(M)$ is $i$-injective.
\end{fact}

\begin{proof}
By hypothesis, the additive functor $\Tor_d^G(\argu,M)$ is left exact.
For $U,V\in G^\sharp$, $V\subseteq U$, the map $i_{U,V}^{\boh_d(M)}$ coincides
with  $\Tor_d^G(\F_p[U\backslash G],M)\to\Tor_d^G(\F_p[V\backslash G],M)$
which is induced by the injection $\F_p[U\backslash G]\to\F_p[V\backslash G]$.
This yields the claim.
\end{proof}

For $i$-injective cohomological Mackey functors one has also the following.

\begin{fact}
\label{fact:iinj}
Let $\caM$ be a $G$-Mackey system of the profinite group $G$,
and let $\boX$ be an $i$-injective cohomological $\caM$-Mackey functor.
Then $\inrst(\boX)=0$ is equivalent to $\boX=0$.
\end{fact}

If $\caM$ contains $G$ and
$i^\boX_{G,U}\colon\boX_G\to\boX_U$ is injective for all $U\in\caM$, then we say that
$\boX$ is {\it terminally $i$-injective}, i.e.,
$\boX$ is terminally $i$-injective if, and only if, the canonical map
\begin{equation}
\label{eq:j}
j_\boX\colon \boX_G\longrightarrow\inrst(\boX)
\end{equation}
is injective.

The cohomological $\caM$-Mackey functor $\boX$ is said to be {\it of type $H^0$} (or to 
{\it satisfy Galois descent}),
if $\boX$ is $i$-injective, and for every normal section $(U,V)$ in $\caM$,
one has $\bok^1(U/V,\rst^\caM_{(U/V)^\circ}(\boX))=0$;
e.g., for $M\in\ob({}_G\dis)$, $\boh^0(M)$ (cf. \eqref{eq:hup0}) is a cohomological $G^\sharp$-Mackey functor of type $H^0$. Moreover, one has the following.

\begin{fact}
\label{fact:hup}
Let $G$ be a profinite group, let $\caM$ be a $G$-Mackey system,
and let $\boX$ be a cohomological $\caM$-Mackey functor of type $H^0$. Then one has
a canonical isomorphism
\begin{equation}
\label{eq:hupires}
\boX\simeq\rst^{G^\sharp}_{\caM}(\boh^0(\inrst(\boX))).
\end{equation}
\end{fact}

From Fact~\ref{fact:vank0} one concludes the following.

\begin{fact}
\label{fact:vank01}
Let $G$ be a profinite group, and let $\caM$ be a $G$-Mackey system.
The cohomological $\caM$-Mackey functor $\boX$ is of type $H^0$ if, and only if,
$\bok^k(U/V,\boX)=0$ for all $U,V\in\caM$, $V\triangleleft U$, $k=0,1$.
\end{fact}

Suppose that $\caM$ contains $G$. If the cohomological $\caM$-Mackey functor $\boX$ is 
$i$-injective and satisfies
$\bok^1(G/U,\boX)=0$ for all $U\in\caM^\natural=M\cap G^\natural$, then we say
that $\boX$ is {\it terminally of type $H^0$}. The following property will turn out to be useful
for our purpose.

\begin{fact}
\label{fact:propterm}
Let $G$ be a pro-$p$ group, let $\caM$ be a $G$-Mackey system containing $G$,
and let $\boX\in\ob(\MC_G(\caM,\cvec))$ be terminally of type $H^0$. Then
for every $U\in\caM^\natural$ one has
$\image(i_{G,U}^\boX)=\soc_{G/U}(\boX_U)$.
\end{fact}

\begin{proof}
As $G$ is a pro-$p$ group, $G/U$ is a finite $p$-group.
Since $\boX$ is terminally of type $H^0$, one has
$\image(i_{G,U}^\boX)=(\boX_U)^{G/U}=\soc_{G/U}(\boX_U)$
(cf. \S\ref{ss:proffinp}).
\end{proof}

\subsection{Hilbert~'90 cohomological Mackey functors}
\label{ss:h90}
Let $G$ be a profinite group, and let $\caM$ be a $G$-Mackey system.
The cohomological $\caM$-Mackey functor $\boX$ is said to be
{\it Hilbert~'90}, if $\boX$ is of type $H^0$ and
for every normal section $(U,V)$ in $\caM$ one has $H^1(U/V,\boX_V)=0$.

Let $G$ be a finite group, and let $\Omega$ be a finite set with a left $G$-action.
Then $\Omega=\bigsqcup_{1\leq j\leq r} \Omega_i$, where $\bigsqcup$ denotes
disjoint union and $\Omega_j$ are the $G$-orbits on $\Omega$.
Let $\omega_j\in\Omega_j$, and let $G_j=\stab_G(\omega_j)$ denote the stabilizer
of $\omega_j$ in $G_j$. For the left $\Z_p[G]$-permutation module $\Z_p[\Omega]$
one has isomorphisms
\begin{equation}
\label{eq:H90ind}
\Z_p[\Omega]\simeq\coprod_{1\leq j\leq r} \Z_p[\Omega_j]\simeq
\coprod_{1\leq j\leq r} \idn_{G_j}^G(\Z_p)\simeq
\coprod_{1\leq j\leq r}\coind_{G_j}^G(\Z_p).
\end{equation}
Hence, by the Eckmann-Shapiro lemma, one has that
\begin{equation}
\label{eq:H90ind2}
H^1(G,\Z_p[\Omega])\simeq \coprod_{1\leq j\leq r} H^1(G_j,\Z_p)=\coprod_{1\leq j\leq r}
\Hom_{\gr}(G_j,\Z_p)=0.
\end{equation}
This fact has the following consequence.

\begin{fact}
\label{fact:h90}
Let $G$ be a profinite group, and let $H$ be a closed
subgroup of $G$. Then the cohomological $G^\sharp$-Mackey functor
$\idn_{H^\sharp}^{G^\sharp}(\boT(H,\Z_p))$ is Hilbert~'90.
\end{fact}

\begin{proof}
Let $\boX=\idn_{H^\sharp}^{G^\sharp}(\boT(H,\Z_p))$.
From the isomorphism $\boX\simeq\boh^0(\coind_H^G(\Z_p))$
(cf. Ex.~\ref{ex:hs}(c)) one concludes that $\boX$ is of type $H^0$.
Let $(U,V)$ be a normal section in $G^\sharp$, and let $W\in G^\sharp$, $W\subseteq V$, which is 
normal in $G$. Then $\boX_W$ is a transitive $\Z_p[G/W]$-permutation module,
i.e., $\boX_W\simeq\Z_p[\Omega]$ for some finite transitive left $G/W$-set $\Omega$.
Let $\Omega=\bigsqcup_{1\leq j\leq r} \Omega_i$ denote the decomposition of $\Omega$ in $V/W$-orbits,
and put $\Xi=\{\,\Omega_i\mid 1\leq i\leq r\,\}$. As $U/W\subseteq N_{G/W}(V/W)$,
$\Xi$ is a left $U/W$-set, and one has an isomorphism of left $\Z_p[U/V]$-modules
$\boX_V\simeq \Z_p[\Xi]$. 
Hence, by \eqref{eq:H90ind2}, $H^1(U/V,\boX_V)=0$ and $\boX$ is Hilbert~'90.
\end{proof}

The Hilbert~'90 property has also the following consequence.

\begin{fact}
\label{fact:h902}
Let $G$ be a profinite group, let $\caM$ be a $G$-Mackey system, and let
$\boX\in\ob(\MC_G(\caM,\prfp))$ be a Hilbert~'90 cohomological $\caM$-Mackey functor
such that $\boX_U$ is torsion free for all $U\in \caM$. 
Then $\boX_{/p}=\boX/p\boX$
is of type $H^0$.
\end{fact}

\begin{proof}
By hypothesis, one has a short exact sequence 
\begin{equation}
\label{eq:h90-1}
\xymatrix{
0\ar[r]&\boX\ar[r]^{p.}&\boX\ar[r]&\boX_{/p}\ar[r]&0
}
\end{equation}
of cohomological $\caM$-Mackey functors.
Let $U,V\in\caM$, $V$ normal in $U$. By Proposition~\ref{prop:secco}(b),
one has $\bok^0(U/V,\boX_{/p})=0$.
Hence by Fact~\ref{fact:vank0}, $\boX_{/p}$ is $i$-injective.
As $\boX$ is Hilbert~'90, the buttom row in the diagram
\begin{equation}
\label{eq:h90-2}
\xymatrix{
0\ar[r]&\boX_U\ar[r]\ar[d]^{\tilde{i}_{U,V}^{\boX}}&\boX_U\ar[r]
\ar[d]^{\tilde{i}^\boX_{U,V}}&(\boX_{/p})_U\ar[r]\ar[d]^{\tilde{i}^{\boX_{/p}}_{U,V}}&0\\
0\ar[r]&\boX_V^{U/V}\ar[r]&\boX_V^{U/V}\ar[r]&(\boX_{/p})_V^{U/V}\ar[r]&0
}
\end{equation}
is exact. As $\boX$ is of type $H^0$, $\tilde{i}_{U,V}^{\boX}$ is an isomorphism.
Hence the snake lemma implies that $\tilde{i}^{\boX_{/p}}_{U,V}$ is an isomorphism.
This yields the claim.
\end{proof}


\subsection{Cohomological Mackey functors of type $H_0$}
\label{ss:comachdown}
Let $G$ be a profinite group, let $\caM$ be a $G$-Mackey system,
and let $\boX$ be a cohomological $\caM$-Mackey functor.
We will say that $\boX$ is {\it $t$-surjective},
if $t^\boX_{V,U}\colon\boX_V\to\boX_U$
is surjective for all $U,V\in\caM$, $V\subseteq U$. 
If $\boY$ is the homomorphic image of a $t$-surjective, cohomological $\caM$-Mackey functor $\boX$,
then $\boY$ is also $t$-surjective.
The following fact is straightforward.

\begin{fact}
\label{fact:vanc0}
Let $G$ be a profinite group, and let $\caM$ be a $G$-Mackey system.
The cohomological $\caM$-Mackey functor $\boX$ is $t$-surjective if, and only if,
$\boc_0(U/V,\boX)=0$ for all $U,V\in\caM$, $V\triangleleft U$.
\end{fact}

From Fact~\ref{fact:elemab}(a) one concludes the following.

\begin{fact}
\label{fact:tfratt}
Let $G$ be a profinite group, let $\caM$ be a $G$-Mackey system,
and $\boX\in\ob(\MC_G(\caM,\prfp))$. Then $\boX$ is $t$-surjective if, and only
if $\boX_{/p}=\boX/p\boX$ is $t$-surjective.
\end{fact}

Applying Pontryagin duality (cf. Rem.~\ref{rem:pontryagin})
to Fact~\ref{fact:iinj} yields the following.

\begin{fact}
\label{fact:tsur}
Let $G$ be a profinite group, let $\caM$ be a $G$-Mackey system,
and let $\boX\in\ob(\MC_G(\caM,\cvec))$ (resp. $\boX\in\ob(\MC_G(\caM,\prfp))$).
Then $\prorst(\boX)=0$ implies $\boX=0$.
\end{fact}

The cohomological $\caM$-Mackey functor $\boX$ is said to be {\it of type $H_0$} (or to 
{\it satisfy Galois codescent}),
if $\boX$ is $t$-surjective, and for every normal section $(U,V)$ in $\caM$, 
one has $\boc^1(U/V,\boX)=0$,
e.g., for $Q\in\ob(\ZGprf)$ the cohomological $G^\sharp$-Mackey functor $\boh_0(Q)$
(cf. \eqref{eq:hdown0})
is of type $H_0$. Furthermore, one has the following.

\begin{fact}
\label{fact:hdown}
Let $G$ be a profinite group, let $\caM$ be a $G$-Mackey system,
and let $\boX\in\ob(\MC_G(\caM,\prfp))$ be a cohomological $\caM$-Mackey functor of type $H_0$. 
Then one has
a canonical isomorphism
\begin{equation}
\label{eq:hdowntres}
\boX\simeq\rst^{G^\sharp}_{\caM}(\boh_0(\prorst(\boX))).
\end{equation}
\end{fact}

\begin{rem}
\label{rem:hdown}
Note that 
$\prorst\colon \MC_G(\caM,\prfp)\to\ZGprf$
is the right-adjoint of 
$\rst^{G^\sharp}_\caM\circ\boh_0\colon\ZGprf\to\MC_G(\caM,\prfp)$.
Let
\begin{equation}
\label{eq:counittres}
\eps\colon \rst^{G^\sharp}_\caM\circ\boh_0\longrightarrow\iid_{\MC_G(\caM,\prfp)}
\end{equation}
denote the counit of the adjunction. Then $\boX$ is of type $H_0$ if, and only if, $\eps_\boX$
is an isomorphism.
\end{rem}

By Fact~\ref{fact:vanc0}, one has also the following.

\begin{fact}
\label{fact:vanc01}
Let $G$ be a profinite group, and let $\caM$ be a $G$-Mackey system.
The cohomological $\caM$-Mackey functor $\boX$ is of type $H_0$ if, and only if,
$\boc_0(U/V,\boX)=\boc_1(U/V,\boX)=0$ for all $U,V\in\caM$, $V\triangleleft U$.
\end{fact}

For projective, profinite left $\ZpG$-modules or $\FpG$-modules the following is true.

\begin{fact}
\label{fact:projct}
Let $G$ be a profinite group, and let $Q$ be a projective,
profinite left $\FpG$-module (resp. $\ZpG$-module).
Then $\boh_0(Q)$ is a cohomological $G^\sharp$-Mackey functor
which is also of type $H^0$.
\end{fact}

\begin{proof}
By definition and Fact~\ref{fact:vanc0}, it suffices to show that $\boh_0(Q)$ is of type $H^0$.
Let $(U,V)$ be a normal section in $G^\sharp$.
As restriction maps projectives to projectives, $\rst^G_U(Q)$ is a
projective, profinite left $\F_p\dbl U\dbr$-module.
Since $\argu_V=\F_p[U/V]\hotimes_U\argu$ is the left-adjoint of the inflation functor $\ifl_{U/V}^U(\argu)$,
it maps projectives to projectives (cf. \cite[Prop.~2.3.10]{weib:hom}).
Therefore $\boh_0(Q)_V=Q_V$ is a projective left $\F_p[U/V]$-module, and
thus $\hH^{-1}(U/V,Q_V)=\hH^0(U/V,Q_V)=0$.
As $\boh_0(Q)$ is of type $H_0$, one has
\begin{equation}
\label{eq:vanproj}
\boc_0(U/V,\rst^{G^\sharp}_{(U/V)^\circ}(\boh_0(Q)))=
\boc_1(U/V,\rst^{G^\sharp}_{(U/V)^\circ}(\boh_0(Q)))=0.
\end{equation}
The exact 6-term sequence \eqref{eq:6term} implies that 
\begin{equation}
\label{eq:vanproj2}
\bok^0(U/V,\rst^{G^\sharp}_{(U/V)^\circ}(\boh_0(Q)))=
\bok^1(U/V,\rst^{G^\sharp}_{(U/V)^\circ}(\boh_0(Q)))=0.
\end{equation}
Thus Fact~\ref{fact:vank01} yields the claim. The proof for 
projective, profinite left $\ZpG$ can be transferred verbatim.
\end{proof}

\begin{rem}
\label{rem:inddefl}
Let $G$ be a profinite group, let $H$ be a closed subgroup of $G$,
and let $N$ be a closed normal subgroup of $G$ containing $H$.
Put $\baG=G/N$.
Let $\boT=\boT(H,\F_p)$ be the cohomological $H^\sharp$-Mackey
functor described in Example~\ref{ex:1}(a), and put
$\boX=\idn_{H^\sharp}^{G^\sharp}(\boT)$. 
Then one has an isomorphism
\begin{equation}
\label{eq:inddefl}
\boX\simeq\boh^0(\coind_H^G(\F_p))
\end{equation}
(cf. Ex.~\ref{ex:hs}(c)). In particular, $\boX$ is of type $H^0$.
Let $\boY=\rst^{G^\sharp}_{\baG^\sharp}(\boX)$ (cf. \eqref{eq:res2}).
For any open subgroup $U$ of $G$ containing $N$, one has
$\boX_U=\F_p[G/U]$ (cf. Ex.~\ref{ex:ind}). Moreover, if $V$ is an open subgroup
of $G$ such that $N\subseteq V\subseteq U$, one has
$H\cap {}^sU=H\cap{}^sV=H$ for all $s\in G$, i.e.,
$t_{V,U}^\boX\colon \boX_V\to\boX_U$ is the canonical map (cf. \eqref{eq:indT3}).
Hence $\boY\simeq\boh_0(\F_p\dbl\baG\dbr)$. In a similar fashion one shows that
\begin{equation}
\label{eq:inddefl2}
\rst^{G^\sharp}_{\baG^\sharp}(\idn_{H^\sharp}^{G^\sharp}(\boT(H,\Z_p)))\simeq
\boh_0(\Z_p\dbl\baG\dbr).
\end{equation}
\end{rem}

\subsection{Injectivity criteria}
\label{ss:injcrit}
For a pro-$p$ group $G$ there are very useful criteria ensuring that
a homomorphism $\phi\colon\boX\to\boY$ of cohomological Mackey functors
with values in the category $\cvec$ or $\prfp$
is injective. For cohomological Mackey functors with values in $\cvec$ one has the following.

\begin{lem}
\label{lem:injmod}
Let $G$ be a pro-$p$ group, and  let $\caM$ be a $G$-Mackey system containing $G$. Suppose that 
for 
$\phi\colon\boX\to\boY\in\mor(\MC_G(\caM,\cvec))$ one has that
\begin{itemize}
\item[(i)] $j_\boY\circ\phi_G\colon \boX_G\to\inrst(\boY)$ is injective  
\textup{(}cf. \textup{\eqref{eq:j})}, and
\item[(ii)] $\boX$ is terminally of type $H^0$.
\end{itemize}
Then $\phi$ is injective.
\end{lem}

\begin{proof}
By hypothesis (i), $i^{\boY}_{G,U}\circ\phi_G\colon\boX_G\to\boY_U$ is injective for
all $U\in\caM$.
First we show that $\rst^\caM_{\caM^\natural}(\phi)\colon\rst^\caM_{\caM^\natural}(\boX)\to
\rst^\caM_{\caM^\natural}(\boY)$ is injective.
For $V\in\caM^\natural$ one has a commutative diagram
\begin{equation}
\label{dia:inj1}
\xymatrix{
\boX_G\ar[r]^{\phi_G}\ar[d]_{i^{\boX}_{G,V}}&\boY_G\ar[d]^{i^{\boY}_{G,V}}\\
\boX_V\ar[r]^{\phi_V}&\boY_V
}
\end{equation}
As $\boX$ is terminally of type $H^0$, $\image(i_{G,V}^{\boX})=\soc_{G/V}(\boX_V)$
(cf. Fact~\ref{fact:propterm}).
Since $\phi_V\circ i^{\boX}_{G,V}=i^{\boY}_{G,V}\circ\phi_G$ is injective,
$\phi_V\vert_{\soc_{G/V}(\boX_V)}\colon\soc_{G/V}(\boX_V)\to\boY_V$ is injective.
In particular, one has 
$\kernel(\phi_V)\cap\soc_{G/V}(\boX_V)=0$ which implies that $\kernel(\phi_V)=0$
(cf. Fact~\ref{fact:soctriv}).
Thus $\phi_V$ is injective.

Let $U\in\ca{M}$.  There exists $V\in\caM^\natural$ such that
$V\subseteq U$, and one has a commutative diagram
\begin{equation}
\label{dia:inj2}
\xymatrix{
\boX_U\ar[r]^{\phi_U}\ar[d]_{i^{\boX}_{U,V}}&\boY_U\ar[d]^{i^{\boY}_{U,V}}\\
\boX_V\ar[r]^{\phi_V}&\boY_V
}
\end{equation}
Thus as $i_{U,V}^{\boX}$ and $\phi_V$ are injective, $\phi_U$ is injective.
This yields the claim.
\end{proof}

Another version of Lemma~\ref{lem:injmod} is the following.

\begin{cor}
\label{cor:injmod2}
Let $G$ be a pro-$p$ group, and let $\caM$ be a $G$-Mackey system containing $G$.
Let $\phi\colon\boX\to\boY\in\mor(\MC_G(\caM,\cvec))$ be a homomorphism of cohomological
$\caM$-Mackey functors satisfying:
\begin{itemize}
\item[(i)] $\phi_G\colon \boX_G\to\boY_G$ is injective;
\item[(ii)] $\boY$ is terminally $i$-injective;
\item[(iii)] $\boX$ is terminally of type $H^0$.
\end{itemize}
Then $\phi$ is injective.
\end{cor}

\begin{proof}
As $\boY$ is terminally $i$-injective, the canonical map
$j_{\boY}\colon\boY_G\to\inrst(\boY)$ is injective.
The claim is therefore a direct consequence of Lemma~\ref{lem:injmod}.
\end{proof}

From Fact~\ref{fact:elemab2} one concludes the following criterion for split-injectivity.

\begin{lem}
\label{lem:inj}
Let $G$ be a pro-$p$ group, and let $\caM$ be a $G$-Mackey system
containing $G$.
Let $\phi\colon \boX\to\boY\in \mor(\MC_G(\caM,\fmod))$ be a homomorphism
of cohomological $\caM$-Mackey functors with values in the category of finitely generated
$\Z_p$-modules with the following properties:
\begin{itemize}
\item[(i)] $\boX_U$ and $\boY_U$ are torsion free $\Z_p$-modules for every $U\in\caM$;
\item[(ii)] $\gr_0(\boX)$ is terminally of type $H^0$;
\item[(iii)] the canonical map 
$j_{\gr_0(\boY)}\circ\gr_0(\phi_G)\colon \gr_0(\boX_G)\to\inrst(\gr_0(\boY))$
is injective.
\end{itemize}
Then $\phi_U\colon\boX_U\to\boY_U$ is split-injective for every $U\in\caM$.
\end{lem}

\begin{proof} 
By hypothesis (ii) and (iii), $\gr_0(\phi)\colon\gr_0(\boX)\to\gr_0(\boY)$ satisfies the
hypothesis of Lemma~\ref{lem:injmod}, and thus is injective.
Let $U\in\caM$ and $k>0$. Then one has a commutative diagram
\begin{equation}
\label{eq:lemtorf}
\xymatrix@C=1.8truecm{
\gr_0(\boX_U)\ar[r]^{\gr_0(\phi_U)}\ar[d]_{t^k.}&\gr_0(\boY_U)\ar[d]^{t^k.}\\
\gr_k(\boX_U)\ar[r]^{\gr_k(\phi_U)}&\gr_k(\boY_U)
}
\end{equation}
where the vertical maps are isomorphisms by hypothesis (i). Hence $\gr_k(\phi_U)$ is injective for all 
$k\geq 0$.
Thus, by Fact~\ref{fact:elemab2}, $\phi_U\colon\boX_U\to\boY_U$ is split-injective.
\end{proof}


\subsection{A dimension theorem}
\label{ss:dimthm}
For a pro-$p$ group $G$
one has the following theorem.

\begin{thm}
\label{thm:dim}
Let $G$ be a pro-$p$ group, let $\caB$ be a basis of neighbourhoods of $1\in G$ 
consisting of open subgroups of $G$,
and let $M$ be a profinite left $\FpG$-module satisfying $H_d(G,M)\not=0$, $d\geq 0$. Then
the following are equivalent.
\begin{itemize}
\item[(i)] $\pdim(M)=d$;
\item[(ii)] $\cores_{G,U}\colon H_d(G,M)\longrightarrow H_d(U,\rst^G_U(M))$
is injective for all $U\in\caB$.
\end{itemize}
\end{thm}

\begin{proof} 
By Fact~\ref{fact:pdim}, it suffices to show that (ii) implies (i).
Since $\caB$ is a basis, (ii) implies that 
$\cores_{G,U}\colon H_d(G,M)\longrightarrow H_d(U,\rst^G_U(M))$
is injective for every open subgroup $U$ of $G$.
Let $(P_\bullet,\der_\bullet,\eps_M)$ be a minimal
projective resolution of $M$ in $\FGprf$ (cf. Fact~\ref{fact:minproj}). Then
$(\boh_0(P_k),\boh_0(\der_k))$
is a chain complex of coho\-mo\-logical $G^\sharp$-Mackey functors
with values in $\cvec$. Furthermore, for all $k\geq 0$
one has the short exact sequences
of cohomological $G^\sharp$-Mackey functors
\begin{equation}
\label{eq:dimth1}
\xymatrix@R=3pt{
0\ar[r]&\kernel(\boh_0(\der_k))\ar[r]^-{\iota^k}&
\boh_0(P_k)\ar[r]^-{\pi^k}&\image(\boh_0(\der_k))\ar[r]&0;\\
0\ar[r]&\image(\boh_0(\der_{k+1}))\ar[r]^-{\alpha^k}&
\kernel(\boh_0(\der_k))\ar[r]^-{\beta^k}&\boh_k(M)\ar[r]&0;
}
\end{equation}
where $\iota_k$ is given by inclusion. 
By Fact~\ref{fact:projct}, $\boh_0(P_k)$ is of type $H^0$ and $H_0$.
Hence $\kernel(\boh_0(\der_k))$ is 
$i$-injective for all $k\geq 0$ (cf. \S\ref{ss:cohmacprof1}).

As $(P_\bullet,\der_\bullet,\eps_M)$ is a minimal projective resolution,
$(\boh_0(P_k)_G,\boh_0(\der_k)_G)$ is a chain complex with values in $\cvec$
with trivial differential (cf. Fact~\ref{fact:trivdiff2}).
Thus $\iota^k_G$ and $\beta^k_G$ are isomorphisms for all $k\geq 0$.

Let $U$ be an open normal subgroup of $G$. As $\boh_0(P_d)$ is of type $H^0$,
and since $\iota^d_G$ is an isomorphism, one has
\begin{equation}
\label{eq:dimth2}
\begin{aligned}
S_U&=\soc_{G/U}(\boh_0(P_d)_U)=\image(i_{G,U}^{\boh_0(P_d)})\\
&=
\image(i_{G,U}^{\kernel(\boh_0(\der_d))})=
\soc_{G/U}(\kernel(\boh_0(\der_d))_U).
\end{aligned}
\end{equation}
Suppose that $U\in\caM$ and consider the commutative diagram
\begin{equation}
\label{eq:dimth3}
\xymatrix{
\kernel(\boh_0(\der_d))_G\ar[r]^{\beta^d_G}\ar[d]_{i_{G,U}^{\kernel(\boh_0(\der_d))}}&
\boh_d(M)_G\ar[d]^{i^{\boh_d(M)}_{G,U}}\\
\kernel(\boh_0(\der_d))_U\ar[r]^{\beta^d_U}&\boh_d(M)_U.
}
\end{equation} 
By hypothesis, $i^{\boh_d(M)}_{G,U}$ is injective. Furthermore,
$i_{G,U}^{\kernel(\boh_0(\der_d))}$ is injective, and
$\beta_G^d$ is an isomorphism.
Thus by \eqref{eq:dimth3},
$\beta^d_U\vert_{S_U}\colon S_U\to \boh_d(M)_U$ must be injective. 
Hence $\kernel(\beta^d_U)$ has trivial intersection with $S_U$, and thus
$\kernel(\beta^d_U)=0$ (cf. Fact~\ref{fact:soctriv}), i.e., $\beta^d_U$ is injective.
From the exact diagram \eqref{eq:dimth1} one concludes that
$\alpha^d_U=0$. In particular, $\prorst(\alpha^d)=0$.

Since $\prorst(\argu)$ is exact (cf. \S\ref{ss:CMprof}), one concludes that
$\prorst(\image(\boh_0(\der_{d+1})))=0$.
Moreover, as $\image(\boh_0(\der_{d+1}))$ is a homomorphic image of $\boh_0(P_{d+1})$,
the cohomological $G^\sharp$-Mackey functor
$\image(\boh_0(\der_{d+1}))$ is $t$-surjective (cf. \S\ref{ss:comachdown}). Hence  
$\image(\boh_0(\der_{d+1}))=0$ (cf. Fact~\ref{fact:tsur}),
i.e., $\boh_0(\der_{d+1})\colon \boh_0(P_{d+1})\to\boh_0(P_d)$ is the $0$-map.
Thus 
\begin{equation}
\label{eq:dimthmder}
\prorst(\boh_0(\der_{d+1}))=\der_{d+1}\colon P_{d+1}\to P_d
\end{equation} 
is the $0$-map
(cf. Rem.~\ref{rem:hdown}).
The minimality of $(P_\bullet,\der_\bullet,\eps_M)$ then implies that $P_{d+k}=0$ for all $k\geq 1$.
\end{proof}


\subsection{Free pro-$p$ groups}
\label{ss:free}
For a pro-$p$ group $G$ we denote by
$\El=\El(G)$ the cohomological $G^\sharp$-Mackey functor $\boh_1(\F_p)$, i.e.,
$\El_U=U^{\ab,\el}$ for all $U\in G^\sharp$.
As a consequence of Theorem~\ref{thm:dim} one has the following
characterization of free pro-p groups.

\begin{cor}
\label{cor:thdim}
Let $G$ be a pro-$p$ group, and let $\caB$ be a basis of neighbourhoods
of $1\in G$ consisting of open subgroups of $G$. 
Then the following are equivalent:
\begin{itemize}
\item[(i)] $G$ is a free pro-$p$ group.
\item[(ii)] $i^{\El}_{U,V}\colon U^{\ab,\el}\to V^{\ab,\el}$ is injective for all $U,V\in G^\sharp$, 
$V\subseteq U$.
\item[(iii)] $i^{\El}_{G,U}\colon G^{\ab,\el}\to U^{\ab,\el}$ is injective for all $U\in\caB$.
\item[(iv)] The canonical map $j_{\El}\colon G^{\ab,\el}\to D_1(\F_p)$ is injective.
\end{itemize}
\end{cor}

\begin{proof}
A non-trivial pro-$p$ group $G$ is free if, and only if, $\ccd_p(G)=1$ (cf. \cite[\S I.4.2, Cor.~2]{ser:gal}).
For a profinite group the cohomological $p$-dimension coincides with the
projective dimension of the trivial profinite $\FpG$-module $\F_p$ (cf. \cite[Prop.~7.1.4]{ribzal:prof}).
The implication (i)$\Rightarrow$(ii) follows from Fact~\ref{fact:pdim},
(ii)$\Rightarrow$(iii) is trivial, and (iii)$\Rightarrow$(i) follows from Theorem~\ref{thm:dim}.
The equivalence (iii)$\Leftrightarrow$(iv) is a direct consequence of the definition of
$D_1(\F_p)$ (cf. \eqref{eq:tateD}).
\end{proof}


\section{Ends and directions}
\label{s:dir}
In this section we will denote by $1_{\Z_p}\in\Z_p$ a fixed
generator of the additive group of the $p$-adic integers.

\subsection{Pro-$p$ groups with an $\F_p$-direction}
\label{ss:Fpdir}
For a pro-$p$ group with an $\F_p$-direction one has the following.

\begin{thm}
\label{thm:Fpdir}
Let $G$ be a pro-$p$ group with an $\F_p$-direction $\facr{G}{\Z_p}{\tau}{\sigma}$.
Let $\Sigma=\image(\sigma)$, and let $N=\cl(\langle\,{}^g\Sigma\mid g\in G\,\rangle)$
denote the closure of the normal closure of $\Sigma$. Then one has the following:
\begin{itemize}
\item[(a)] $N$ is a free pro-$p$ group.
\item[(b)] Let $\baG=G/N$. Then one has a canonical isomorphism 
\begin{equation}
\label{eq:spec1}
N/\Phi(N)\simeq\F_p\dbl\baG\dbr
\end{equation}
of profinite, left $\F_p\dbl\baG\dbr$-modules.
\item[(c)] The extension of pro-$p$ groups
\begin{equation}
\label{eq:spec2}
\xymatrix{
\triv\ar[r]&N/\Phi(N)\ar@{^(->}[r]&G/\Phi(N)\ar@{-{>>}}[r]&\baG\ar[r]&\triv
}
\end{equation}
splits.
\end{itemize}
\end{thm}

\begin{proof} (a) Let $\boT=\boT(\Sigma,\F_p)$ (cf. Ex.~\ref{ex:1}(a)), and let
$\phi\colon\boT\to\El(\Sigma)=\prorst^{G^\sharp}_{\Sigma^\sharp}(\El(G))$
(cf. Fact~\ref{fact:tres})
be the isomorphism of cohomological $\Sigma^\sharp$-Mackey functors induced by $s=\sigma(1_{\Z_p})$, i.e.,
for $\Xi\in\Sigma^\sharp$, $|\Sigma:\Xi|=p^h$, one has
$\phi_\Xi(1_{\F_p})=s^{p^h}\Phi(\Xi)$.

Let $\tphi\colon\idn_{\Sigma^\sharp}^{G^\sharp}(\boT)\to\El(G)$
be the mapping induced by $\phi$ (cf. \eqref{eq:rightad}), i.e., with the
notation established in Example~\ref{ex:ind}(a) one has for $U\in G^\sharp$, 
$G=\bigsqcup_{r\in\caR} \Sigma r U$, $\caR\subseteq G$ a system of coset
representatives for $\Sigma\backslash G/U$
and $|\Sigma\colon \Sigma\cap rUr^{-1}|=p^{h^r}$ that
\begin{equation}
\label{eq:thmFptphi}
\tphi_U(\Sigma r U)=(r^{-1}s^{p^{h_r}} r)\Phi(U).
\end{equation}
As $\idn_{\Sigma^\sharp}^{G^\sharp}(\boT)\simeq\boh^0(\coind_\Sigma^G(\F_p))$
(cf. \eqref{eq:indTS}), $\idn_{\Sigma^\sharp}^{G^\sharp}(\boT)$ is of type $H^0$.
Moreover, $j_{\El}(\tphi_G(\Sigma 1G))=j_{\El}(s\Phi(G))\not=0$, i.e.,
$j_{\El}\circ\tphi_G$ is injective. Thus by Lemma~\ref{lem:injmod}, $\tphi$ is injective.
Put $\boZ=\image(\tphi)\subseteq \El(G)$.
Then $\boZ$ is an $i$-injective cohomological $G^\sharp$-Mackey subfunctor of $\El(G)$
(cf. \S\ref{ss:cohmacprof1}).

Since $G$ is a pro-$p$ group,
$\caM=\{\,V\in N^\sharp\mid V\ \text{normal in}\ G\,\}$
is a normal $N$-Mackey basis (cf. \S\ref{ss:macksys}).
Suppose that $N$ is not free, i.e., $\ccd_p(N)>1$.
Thus by Corollary \ref{cor:thdim}(iii),
there exists an element $n\in N\setminus\Phi(N)$, and an open subgroup
$W_\circ\in\caM$ such that $i_{N,W_\circ}^{\El}(n\Phi(N))=0$.
Hence $i_{N,W}^{\El}(n\Phi(N))=0$ for all $W\in W_\circ^\sharp$.
Since $n\not\in\Phi(N)$, there exists an open subgroup $V_\circ$ of $G$
containing $N$ such that $n\not\in\Phi(V_\circ)$ (cf. \cite[Prop.~2.8.9]{ribzal:prof}).
Hence $n\not\in\Phi(V)$ for all open subgroups $V$ of $V_\circ$ containing $N$.

Let $U$ be an open normal subgroup of $G$, such that $U\cap N\subseteq W_\circ$
and $U\subseteq V_\circ$. In particular, $NU\subseteq V_\circ$.
As $NU/U$ and $N/(N\cap U)$ are canonically isomorphic, one has a commutative diagram
\begin{equation}
\label{eq:fr2b}
\xymatrix{
N/\Phi(N)\ar[r]^-{\alpha}\ar[d]_{i^{\El(N)}_{N,N\cap U}}\ar@{..>}[1,1]^{\gamma}& 
NU/\Phi(NU)\ar[d]^{i^{\El(G)}_{NU,U}}\\
(N\cap U)/\Phi(N\cap U)\ar[r]^-{\beta}&U/\Phi(U),
}
\end{equation}
where $\alpha$ and $\beta$ are the canonical maps. Let $\gamma$ denote 
the diagonal map. 
Since $NU\subseteq V_\circ$, $\alpha(n.\Phi(N))\not=0$. 
As $N$ is generated as normal subgroup by $\Sigma$, one has 
$\image(\alpha)=\boZ_{NU}$. As
$i^{\El(G)}_{NU,U}\vert_{\boZ_{NU}}=i^\boZ_{NU,U}$ is injective,
one has $\gamma(n\Phi(N))\not=0$.
Since $N\cap U$ is an open subgroup of $W_\circ$,
one has $i^{\El(N)}_{N,N\cap U}(n\Phi(N))=0$; in particular, 
$\gamma(n\Phi(N))=0$, a contradiction. Thus $N$ is a free pro-$p$ group.

\noindent
(b) As $N$ is generated as normal subgroup by $\Sigma$, one has
an isomorphism of profinite left $\F_p\dbl\baG\dbr$-modules
\begin{equation}
\label{eq:thmFpmod}
N/\Phi(N)=\textstyle{\varprojlim_{N\subseteq U\subseteq_\circ G} \boZ_U}.
\end{equation}
Let $\boY=\rst^{G^\sharp}_{\baG^\sharp}(\boZ)$.
Then $\boY\simeq\boh_0(\F_p\dbl \baG\dbr)$ (cf. Rem.~\ref{rem:inddefl}).
In particular, one has an isomorphism
$\textstyle{\varprojlim_{N\subseteq U\subseteq_\circ G} \boZ_U}
=
\prorst_{\baG^\sharp}(\boY)\simeq\F_p\dbl \baG\dbr$.
(cf. Rem.~\ref{rem:hdown}).
 
\noindent
(c) Since $\FpbG$ is a profinite, projective left $\FpbG$-module,
$\boY=\rst^{G^\sharp}_{\baG^\sharp}(\boZ)$ is also of type $H^0$
(cf. Fact~\ref{fact:projct}), i.e., $\boY\simeq\boh^0(\coind_{\triv}^{\baG^\sharp}(\F_p))$.
The functor
\begin{equation}
\label{eq:ireslad}
\inrst_{\baG^\sharp}\colon \MC_{\baG}(\baG^\sharp,\pdis)\longrightarrow
{}_{\FpbG}\dis
\end{equation}
is the left-adjoint of the functor
\begin{equation}
\label{eq:h0rad}
\boh^0(\argu)\colon{}_{\FpbG}\dis\longrightarrow\MC_{\baG}(\baG^\sharp,\pdis).
\end{equation}
As $\inrst_{\baG^\sharp}$ is exact, $\boh^0(\argu)$ is mapping
injective, discrete left $\FpbG$-modules to injectives functors
in $\MC_{\baG}(\baG^\sharp,\pdis)$ (cf. \cite[Prop.~2.3.10]{weib:hom}). In particular, $\boY$ is injective
in $\MC_{\baG}(\baG^\sharp,\pdis)$.
Since $\boA=\rst^{G^\sharp}_{\baG^\sharp}(\El(G))\in \ob(\MC_{\baG}(\baG^\sharp,\pdis))$,
there exists a cohomological $\baG^\sharp$-subfunctor $\boB$ of $\boA$
such that $\boA=\boB\oplus\boY$. 

Let $\bcM$ be a normal $\baG$-Mackey basis, and let
$\pi\colon G\to\baG$ denote the canonical projection. 
For $\baU\in\bcM$ let $U=\{\,u\in G\mid\pi(u)\in \baU\,\}$, and put $\tG_{\baU}=G/\Phi(U)$.
Then one has canonical projections $\tpi_{\baU}\colon \tG_{\baU}\to G/U$ for all
$\baU\in\bcM$. Moreover, for $\baV\in\bcM$ with
$\baV\subseteq\baU$ one has a commutative diagram
\begin{equation}
\label{eq:frattsplit1}
\xymatrix{
\triv\ar[r]&\boA_{\baV}\ar[d]_{t^\boA_{\baV,\baU}}\ar[r]^{\tiota_{\baV}}&\tG_{\baV}\ar[r]^{\tpi_{\baV}}\ar[d]_{\tmu_{\baV,\baU}}
&G/V\ar[d]^{\nu_{V,U}}\ar[r]&\triv\\
\triv\ar[r]&\boA_{\baU}\ar[r]^{\tiota_{\baU}}&\tG_{\baU}\ar[r]^{\tpi_{\baU}}&G/U\ar[r]&\triv
}
\end{equation}
with exact rows, where
$\tmu_{\baV,\baU}\colon G_{\baV}\to G_{\baU}$ and $\nu_{V,U}\colon G/V\to G/U$
are the canonical maps. We also assume that $\tiota_{\baU}$ is given by inclusion.
Put $G_{\baU}=\tG_{\baU}/\boB_{\baU}$, and let $\pi_{\baU}\colon\colon G_{\baU}\to G/U$
denote the canonical projection. Then one has a commutative diagram
\begin{equation}
\label{eq:frattsplit2}
\xymatrix{
\bt_{\baV}\colon&\triv\ar[r]&\boY_{\baV}\ar[d]_{t^\boY_{\baV,\baU}}\ar[r]&G_{\baV}\ar[r]^{\pi_{\baV}}\ar[d]_{\mu_{\baV,\baU}}
&G/V\ar[d]^{\nu_{V,U}}\ar[r]&\triv\\
\bt_{\baU}\colon&\triv\ar[r]&\boY_{\baU}\ar[r]&G_{\baU}\ar[r]^{\pi_{\baU}}&G/U\ar[r]&\triv
}
\end{equation}
with exact rows, where $\mu_{\baV,\baU}\colon G_{\baV}\to G_{\baU}$ is the induced map.
Let $\rho_{\baU}\colon G\to G_{\baU}$ denote the canonical projection, and put
$K_{\baU}=\kernel(\rho_{\baU})$. By construction, one has $\Phi(U)\subseteq K_{\baU}\subseteq U$.
In particular, for $K=\bigcap_{\baU\in\bcM} K_{\baU}$, one has
\begin{equation}
\label{eq:frattsplit3}
\Phi(N)\subseteq K\subseteq N
\end{equation}
(cf. \cite[Prop.~2.8.9]{ribzal:prof}). By construction, one has $N/K\simeq\FpG$, and - thus by (b) -
 $K=\Phi(N)$. Let
 \begin{equation}
 \label{eq:frattsplit4}
 \xymatrix{
 \bos\colon\qquad
 \triv\ar[r]& N/\Phi(N)\ar[r]& G/\Phi(N)\ar[r]& \baG\ar[r]&\triv
 }
 \end{equation}
 denote the short exact sequence associated to the extension $G/\Phi(N)\to\baG$,
 and let $[\bos_{\baU}]$ denote the image of $[s]$ in $H^2_{\cts}(\baG,\F_p[\baG/\baU])$.
 Then - by construction - $[\bos_{\baU}]=\ifl_{G/U}^{\baG}([\bt_{\baU}])$, where
 $\ifl_{G/U}^{\baG}(\argu)$ is the inflation map
 \begin{equation}
 \label{eq:frattsplit5}
 \ifl_{G/U}^{\baG}(\argu)\colon H^2(G/U,\F_p[\baG/\baU])\longrightarrow H^2_{\cts}(\baG,\F_p[\baG/\baU]),
 \end{equation}
 and $\bt_{\baU}$ is given as in \eqref{eq:frattsplit2}.
 As $\F_p[\baG/\baU]$ is an injective $\F_p[\baG/\baU]$-module, one has $[\bt_{\baU}]=0$,
 and therefore $[\bos_{\baU}]=0$. Hence, by Proposition~\ref{prop:splitex}, one has $[\bos]=0$,
 and this yields the claim.
\end{proof}


\subsection{Special extensions of pro-$p$ groups}
\label{ss:spec}
The following theorem shows that the splitting of
a special pro-$p$ extension implies the decomposition of 
the extension group $G$ as a free product.

\begin{thm}
\label{thm:comfree}
Let $\xymatrix{
\triv\ar[r]&N\ar[r]^{\iota}&G\ar[r]_{\pi}\ar[r]&\baG\ar[r]\ar@/_/[l]_\theta&\triv}$
be a split special pro-$p$ extension, i.e., $\pi\circ \theta=\iid_{\baG}$.
Then there exists a closed free pro-$p$ subgroup $F\subseteq N$ such that
$G$ is isomorphic to $F\coprod \baG$.
\end{thm}

\begin{proof} As $\baG$ is a pro-$p$ group, every projective profinite left $\FpbG$-module is free, 
i.e., there exists a profinite set $\euY$ and an isomorphism
of left $\FpbG$-modules
\begin{equation}
\label{eq:euX}
\psi\colon \FpbG\hotimes \F_p\dbl\euY\dbr\longrightarrow N/\Phi(N).
\end{equation}
Let $s\colon N/\Phi(N)\to N$ be a continuous section of the canonical projection
(cf. \cite[Prop.~I.1.2]{ser:gal}). Then $\euX=s(\psi(1\hotimes\euY))\subset N$
is a closed subset and thus profinite. Let $F(\euX)$ be the free pro-$p$ group
defined over the profinite set $\euX$. Then one has a canonical homomorphism of pro-$p$
groups $\phi\colon F(\euX)\to N$. Put $F=\image(\phi)$. Then - as $F$ is a closed subgroup
of $N$ - $F$ is free.
By construction, the induced
morphism of profinite $\F_p$-vector spaces 
\begin{equation}
\label{eq:euXvec}
\tphi\colon F(\euX)/\Phi(F(\euX))\longrightarrow F/\Phi(F)
\end{equation}
is an isomorphism. Thus $\phi_\circ\colon F(\euX)\to F$ is an isomorphism, 
and $F$ is free over the profinite set $\phi_\circ(\euX)$. Let $\iota_\circ\colon F\to G$
be the mapping induced by inclusion.
Then one has a homomorphism of pro-$p$ groups
$\beta=\iota_\circ\coprod\theta\colon F\coprod \baG\longrightarrow G$.
Put $M=\cl(\langle {}^gF\mid g\in G\rangle)$ and consider the commutative diagram with exact rows
\begin{equation}
\label{eq:pfreepdia}
\xymatrix{
\triv\ar[r]&M\ar[d]^{\alpha}\ar[r]&F\coprod \baG\ar[r] \ar[d]^\beta&\baG\ar[r]\ar@{=}[d]&\triv\\
\triv\ar[r]&N\ar[r]&G\ar[r]&\baG\ar[r]&\triv.
}
\end{equation}
By construction, the induced map $\alpha_\circ\colon M/\Phi(M)\to N/\Phi(N)$ is an isomorphism
of profinite $\F_p$-vector spaces. As $M$ and $N$ are free, $\alpha$ is an isomorphism.
From the snake lemma one concludes that $\beta$ is an isomorphism.
\end{proof}


\subsection{The number of $\F_p$-ends}
\label{ss:Fpends}
Let $G$ be a profinite group, and let $H\subseteq G$ be a closed subgroup of $G$.
Then one can define an {\it index} $|G:H|$ of $H$ in $G$ 
which is a {\it supernatural number} (cf. \cite[Prop.~2.3.2(c)]{ribzal:prof}).
The fixed points of the transitive permutation modules $\F_p\dbl G/H\dbr$
and $\Z_p\dbl G\dbr$ are related to this index as follows.

\begin{prop}
\label{prop:fixperm}
Let $G$ be a profinite group, and let $H\subseteq G$ be a closed subgroup of $G$. Then
\begin{equation}
\label{eq:fixperm}
\Hom_G(\F_p,\F_p\dbl G/H\dbr)\simeq
\begin{cases}
\F_p&\qquad\text{if $p^\infty$ does not divide $|G:H|$,}\\
\hfil 0\hfil&\qquad\text{if $p^\infty$ divides $|G:H|$,}
\end{cases}
\end{equation}
and
\begin{equation}
\label{eq:fixperm2}
\Hom_G(\Z_p,\Z_p\dbl G/H\dbr)\simeq
\begin{cases}
\Z_p&\qquad\text{if $p^\infty$ does not divide $|G:H|$,}\\
\hfil 0\hfil&\qquad\text{if $p^\infty$ divides $|G:H|$.}
\end{cases}
\end{equation}
\end{prop}

\begin{proof} Let $U$ and $V$ be open normal subgroups in $G$, $V\subseteq U$.
Then one has a commutative diagram
\begin{equation}
\label{eq:fixperm3}
\xymatrix{
\Z_p\ar[d]\ar[r]^{\beta}&\Z_p\ar[d]\\
\Z_p[G/HV]^G\ar[r]^{\alpha_{V,U}}&\Z_p[G/HU]^G\\
}
\end{equation} 
where the vertical maps are isomorphisms, 
$\alpha_{V,U}$ is the canonical map, and
$\beta$ is given by multiplication with $|HU/HV|$.
This shows \eqref{eq:fixperm2}.
The isomorphism \eqref{eq:fixperm} follows by a similar argument.
\end{proof}

For a pro-$p$ group $G$, A.A.~Korenev defined in \cite{kor:ends} the number of
$\F_p$-ends by
\begin{equation}
\label{eq:numFpends}
\EE(G)=1-\dim_{\F_p}(H^0_{\cts}(G,\FpG))+\dim_{\F_p}(H^1_{\cts}(G,\FpG)).
\end{equation}
If $G$ is countably based, one has an isomorphism
\begin{equation}
\label{eq:Fpendgr}
D_1(\F_p)^\vee\simeq\textstyle{\varprojlim_{U\subseteq_\circ G} H^1(U,\F_p)}\simeq
H^1_{\cts}(G,\FpG).
\end{equation}
Hence, by Proposition~\ref{prop:fixperm}, for an infinite finitely generated pro-$p$ group $G$ one has
$\EE(G)=1+\dim(D_1(\F_p))$.
Let $G$ be a finitely generated pro-$p$ group. Then $G$ is called 
{\it FAb}\footnote{This is an abbreviation 
for {\it finite abelianizations}.}
if $U^{\ab}=U/\cl([U,U])$
is finite for every open subgroup $U$ of $G$.
The following theorem has been proved in \cite{wei:fab}.

\begin{thm}
\label{thm:sel}
Let $G$ be a FAb pro-$p$ group. Then $H^1_{\cts}(G,\FpG)=0$.
In parti\-cular, if $G$ is infinite, then $\EE(G)=1$.
\end{thm}

From Theorem~\ref{thm:sel} one concludes the following.

\begin{thm}
\label{thm:mend}
Let $G$ be a finitely generated pro-$p$ group with $\EE(G)>1$.
Then there exists an open normal subgroup $U$ of $G$ with an 
$\F_p$-direction $\facr{U}{\Z_p}{\tau}{\sigma}$.
\end{thm}

\begin{proof}
By Theorem~\ref{thm:sel}, there exists an open subgroup $U$ of $G$ such that
$U^{\ab}$ is an infinite, finitely generated $\Z_p$-module.
Since $V^{\ab}$ is infinite for any open subgroup $V$ of $U$, the open subgroup
$U$ can be chosen to be normal in $G$.
As $D_1(\F_p)=\varinjlim_{U\subseteq_\circ G} U^{\ab,\el}$,
we may also assume that the map $j_{\El,U}\colon U^{\ab,\el}\to D_1(\F_p)$
is non-trivial.
There exists a non-trivial torsion free $\Z_p$-submodule $C$ of $U^{\ab}$ satisfying
$U^{\ab}=\tor(U^{\ab})\oplus C$. Let $\eta\colon U^{\ab}\to U^{\ab,\el}$
denote the canonical map. We may consider two cases separately.

\noindent
{\bf Case 1:} $j_{\El,U}(\eta(C))\not=0$. In this case there exists $c\in C$ such that $j_{\El,U}(\eta(c))\not=0$.
As $c\not\in pC\subseteq\Phi(U^{\ab})$, $\Z_pc$ is a direct summand of $C$ isomorphic to $\Z_p$,
and thus also a direct summand of $U^{\ab}$. 
Let $\tci\in U$ be such that $c=\tci\cl([U,U])$, and
let $V$ be a $\Z_p$-submodule of $U^{\ab}$
such that $U^{\ab}=\Z_pc\oplus V$. 
Define the homomorphism $\ttau\colon U^{\ab}\to\Z_p$ by
$\ttau(c)=z$, $V=\kernel(\ttau)$, and let $\tau\colon U\to\Z_p$ be the induced map; 
define $\sigma\colon\Z_p\to U$ by $\sigma(z)=\tci$. It is straightforward to verify that
$\facr{U}{\Z_p}{\tau}{\sigma}$ is an $\F_p$-direction.

\noindent
{\bf Case 2:} $j_{\El,U}(\eta(C))=0$. By hypothesis, there exists an element
$x\in \tor(U^{\ab})$ satisfying $j_{\El,U}(\eta(x))\not=0$,
and $c\in C$, $c\not\in pC$. Let $C_\circ\subseteq C$ be such that $C=\Z_pc\oplus C_\circ$.
Then $y=x+c$ is an element
generating a submodule isomorphic to $\Z_p$, and $C^\prime=\Z_py\oplus C_\circ$
is a complement of $\tor(U^{\ab})$ in $U^{\ab}$.
By construction, $j_{\El,U}(\eta(C^\prime))\not=0$ and one can proceed as in Case 1.
\end{proof}

\begin{rem}
\label{rem:nFpends}
Let $G$ be a finitely generated, infinite pro-$p$ group.
Then either $\EE(G)=1$, or $G$ contains an open subgroup
$U$ with an $\F_p$-direction $\facr{U}{\Z_p}{\tau}{\sigma}$
(cf. Thm.~\ref{thm:mend}).
Put $\Sigma=\image(\sigma)$.
If $|U:\Sigma|<\infty$, then $G$ is virtually cyclic, and - as
$\dim(D_1(p))=1$ - one has $\EE(G)=2$.
Assume that $|U:\Sigma|=\infty$.
By the proof of Theorem~\ref{thm:Fpdir}(a),
$D_1(\F_p)$ contains the subgroup $\inrst_{U^\sharp}(\boZ)$ 
which is isomorphic to 
$\coind_\Sigma^U(\F_p)$. Hence $\dim(D_1(\F_p))=\infty$,
and therefore $\EE(G)=\infty$.
Hence Theorem~\ref{thm:Fpdir} and Theorem~\ref{thm:mend}
provide an alternative proof of the fact that the only possible 
values for $\EE(G)$ are $0$, $1$, $2$ and $\infty$ (cf. \cite{kor:ends}).
\end{rem}

Finally we obtain the following structural result for finitely
generated pro-$p$ groups with infinitely many ends
(cf. Thm.~\ref{thm:Fpdir}, Thm.~\ref{thm:mend}).

\begin{thm}
\label{thm:infends}
Let $G$ be a finitely generated pro-$p$ group satisfying $\EE(G)=\infty$.
Then $G$ contains an open normal subgroup $U$ and a closed subgroup $N$, $N\not=U$,
with the following properties:
\begin{itemize}
\item[(a)] $N$ is normal in $U$;
\item[(b)] $N$ is a non-trivial free pro-$p$ subgroup; 
\item[(c)] $N^{\ab,\el}\simeq\F_p\dbl U/N\dbr$ as profinite left $\F_p\dbl U/N\dbr$-module;
\item[(d)] the extension $\triv\to N^{\ab,\el}\to U/\Phi(N)\to U/N\to\triv$ splits.
\end{itemize}
\end{thm}

\providecommand{\bysame}{\leavevmode\hbox to3em{\hrulefill}\thinspace}
\providecommand{\MR}{\relax\ifhmode\unskip\space\fi MR }
\providecommand{\MRhref}[2]{%
  \href{http://www.ams.org/mathscinet-getitem?mr=#1}{#2}
}
\providecommand{\href}[2]{#2}

\end{document}